%% file: hi_infinitude_revised.tex
\documentclass[11pt]{article}
\usepackage[margin=1in]{geometry}
\usepackage{amsmath,amssymb,amsthm,mathtools}
\usepackage{hhline}
\usepackage{graphicx}
\usepackage[T1]{fontenc}
\usepackage{lmodern,microtype,booktabs}
\usepackage[colorlinks=true,linkcolor=blue,citecolor=blue,urlcolor=blue,
  bookmarksnumbered=true]{hyperref}
\hypersetup{pdftitle={Cyclic Haagerup--Izumi fusion categories at every odd order},
  pdfsubject={Hyperbolic gamma coefficients and cyclic Haagerup--Izumi fusion categories}}
\numberwithin{equation}{section}
\newtheorem{theorem}{Theorem}[section]
\newtheorem{proposition}[theorem]{Proposition}
\newtheorem{lemma}[theorem]{Lemma}

\theoremstyle{definition}

\theoremstyle{remark}
\newtheorem{remark}[theorem]{Remark}
\DeclareMathOperator{\Res}{Res}

\newcommand{\Z}{\mathbb Z}
\newcommand{\Q}{\mathbb Q}
\newcommand{\C}{\mathbb C}
\newcommand{\R}{\mathbb R}
\newcommand{\one}{\mathbf1}
\newcommand{\ind}[1]{\mathbf1_{\{#1\}}}
\allowdisplaybreaks[2]
\title{Cyclic Haagerup--Izumi fusion categories\\at every odd order}
\author{Tzu-Chen Huang}
\date{September 14, 2026}

\begin{document}
\maketitle
\begin{abstract}
We construct a complex spherical fusion category with cyclic
Haagerup--Izumi fusion rules for every odd $n\ge3$, and deduce
pseudo-unitary existence. The proof has four parts: explicit real
coefficients and their quadratic identities; a contour calculation
for a range of cubic Fourier coefficients; algebraic completion of
all cubics; and categorical reconstruction. Matrix inversion supplies
reflection, and an extension of the dimension-field automorphism
supplies the positive-dimensional category.
\end{abstract}

\begingroup
\small
\makeatletter
\renewcommand*\l@section{\@dottedtocline{1}{0pt}{1.8em}}
\makeatother
\tableofcontents
\endgroup
\clearpage

\input{proof_revised/introduction}
\section{Overview}
\label{sec:overview}

Let $G=\Z/n\Z$. The cyclic Haagerup--Izumi based ring has basis
$\alpha^a,\rho_a$ ($a\in G$) and products
\begin{align}
 \alpha^a\alpha^b&=\alpha^{a+b}, &
 \alpha^a\rho_b&=\rho_{a+b}, &
 \rho_a\alpha^b&=\rho_{a-b},\label{intro:fusion1}\\
 \rho_a\rho_b&=\alpha^{a-b}+\sum_{c\in G}\rho_c.
 \label{intro:fusion2}
\end{align}
Put $d=(n+\sqrt{n^2+4})/2$, the noninvertible
Frobenius--Perron dimension, and $\delta=-d^{-1}$.

\begin{theorem}\label{intro:main}
For every odd $n\ge3$, the explicit matrix \eqref{foundation:core}
satisfies the Evans--Gannon reconstruction equations
\cite[Eqs.~(4.7)--(4.10)]{EG} with their parameter $\omega=1$.
It gives a complex spherical fusion category with
\eqref{intro:fusion1}--\eqref{intro:fusion2} and noninvertible
categorical dimension $\delta$. The same based ring also admits a
pseudo-unitary categorification. In particular, there are infinitely
many pairwise inequivalent categories of each kind.
\end{theorem}

The argument proceeds as follows.
\begin{enumerate}
\setlength{\itemsep}{2pt}
\item \textbf{Construct the matrix and prove the quadratics.}
Sample the hyperbolic gamma function, evaluate each column's finite
Fourier transform, and use reflection to obtain its convolution inverse.
\item \textbf{Prove a range of cubic Fourier identities.}
Express the cubic residual as a difference of product transforms.
A four-term contour identity makes $n-g+1$ Fourier coefficients vanish
in the ordered index range.
\item \textbf{Complete the cubics algebraically.}
Reduce degenerate indices to the quadratics. Permutations and polynomial
interpolation give every ordered cubic; a normal form and matrix
inversion give every remaining cubic.
\item \textbf{Reconstruct the categories.}
The quartics follow by contracting the cubic residual. Apply
Evans--Gannon reconstruction, and then change the embedding of the
dimension field to obtain positive categorical dimensions.
\end{enumerate}

\paragraph{Formal verification.}
An accompanying Lean~4 development using Mathlib constructs the
hyperbolic gamma function at the prescribed periods and verifies
the linear, quadratic, cubic, and quartic coefficient identities.
Its theorem \texttt{EGEquations\_ruijsenaars} applies to the explicit
coefficients, with no assumed special-function identities.
The categorical reconstruction and the dimension-field embedding
argument are given in Section~\ref{tail:reconstruction}; these steps
are outside the formalization.

\subsection*{Notation}
Indices lie in $G$ unless integer representatives are specified;
unqualified sums run over $G$. We write $\ind{P}$ for the indicator
of $P$. The $W_a$ are real coefficients and the $\mathsf A_{a,b}$
are matrix entries. A \emph{phase factor} has absolute value one.
The identity $W_aW_{-a}=\delta<0$ for $a\ne0$ is an algebraic
reciprocal relation.
\begin{center}
\small
\begin{tabular}{ll}
\toprule
Symbol & Meaning \\
\midrule
$d,\delta$ & positive Frobenius--Perron and negative categorical dimensions \\
$\tau=\delta-1$ & matrix normalization \\
$\omega_1,\omega_2;\ \Omega=\omega_1+\omega_2$ & positive periods and their sum \\
$F(z),W_a$ & analytic sampling function and real coefficients \\
$\mathsf B,\mathsf A=\tau^{-1}\mathsf B$ & coefficient matrix and reconstruction matrix \\
$u_p(a)=\mathsf B_{a,p}$ & column $p$ \\
$Q,C,\mathcal C=\tau^3 C$ & quadratic, cubic, and rescaled cubic residuals \\
$\theta,\theta_L,\theta_R$ & real Fourier frequencies \\
\bottomrule
\end{tabular}
\end{center}
\clearpage

\part{Construction and quadratic identities}
\label{part:quadratics}
\input{proof_revised/analytic_setup}
\input{proof_revised/equations}

\part{The contour calculation for the cubics}
\label{part:contours}
\input{proof_revised/contour_cubics}

\part{Algebraic completion of the cubics}
\label{part:completion}
\input{proof_revised/cubic_completion}

\part{Categorical reconstruction}
\label{part:reconstruction}
\input{proof_revised/reconstruction}
\section*{Acknowledgement}
The author thanks Terry Gannon, David E. Evans, Ying-Hsuan Lin, and Yuji Tachikawa for valuable exchanges. Some computations were carried out using Magma\cite{Magma}, made available through support from the Simons Foundation. The author thanks Dam T. Son for helping with Magma availability at the University of Chicago.

\noindent\textbf{AI-use disclosure.}
AI systems were used for the computational searches, development and
checking of arguments. The author takes responsibility for the accuracy
and integrity of the final contents. Working transcripts, scripts for
numerical checks, and the Lean formalization of the coefficient identities
are available upon request.

\end{document}

%% file: proof_revised/introduction.tex
\section{Introduction}
\label{sec:introduction}

\subsection*{Origin of the problem}

Haagerup's 1993 search for principal graphs of irreducible finite-depth
subfactors with index in $(4,3+\sqrt2)$ produced a short list of
exceptional candidates beyond the ADE region \cite{Haagerup}. Asaeda and
Haagerup then constructed hyperfinite subfactors at the two indices
$(5+\sqrt{13})/2$ and $(5+\sqrt{17})/2$ by explicit connection
calculations \cite{AH}. Peters later reconstructed the first of these
inside a graph planar algebra \cite{Peters}, and the completed
classification of standard invariants up to index~$5$ confirmed how
isolated these examples are \cite{JMS}. The Haagerup subfactor has the
smallest index among irreducible finite-depth subfactors above index~$4$,
and its fusion categories are basic examples of exotic symmetry.

One even part of the Haagerup subfactor has six simple objects: three
invertibles forming $\Z/3\Z$ and three noninvertibles permuted freely by
them. Izumi realized this category by endomorphisms of a Cuntz algebra
and observed that the multiplication table makes sense for any finite
abelian group $G$ in place of $\Z/3\Z$ \cite{Izumi2}. The resulting
\emph{Haagerup--Izumi} fusion ring has $2n$ basis elements
$\alpha^a,\ \alpha^a\rho$ ($a\in G$, $n=|G|$), with $G$ acting
freely on the noninvertible orbit, $\alpha^a\rho=\rho\alpha^{-a}$, and
\[
 \rho^2=\one\oplus\bigoplus_{a\in G}\alpha^a\rho .
\]
The full table is displayed in
\eqref{intro:fusion1}--\eqref{intro:fusion2}. The Frobenius--Perron
dimension of every noninvertible simple is the positive root
$d=(n+\sqrt{n^2+4})/2$ of $d^2=nd+1$, and the Frobenius--Perron
dimension of the category is $n(1+d^2)$. When $n>1$ is odd the ring is
noncommutative, so none of these
categories is braided; the modular objects attached to them are their
Drinfeld centers. For $n=1$ the ring is $\rho^2=\one\oplus\rho$, whose
two categorifications are the unitary Fibonacci category, with
$\dim\rho=d$, and the nonunitary Yang--Lee category, with
$\dim\rho=-1/d$ \cite{Ostrik2}. For every odd $n\ge3$, the present
paper constructs a spherical categorification with noninvertible
dimension $-1/d$ and a pseudo-unitary categorification with dimension
$d$; a unitary structure on the latter is not established.

\begin{table}[htbp]
\centering
\small
\setlength{\tabcolsep}{4pt}
\begin{tabular}{|p{0.20\linewidth}|p{0.32\linewidth}|p{0.37\linewidth}|}
\hline
Framework & Construction or condition & Scope of comparison\\
\hhline{|=|=|=|}
Izumi \cite{Izumi2}
& Cuntz-algebra construction with a $Q$-system on $\one\oplus\rho$.
& Unitary categories satisfying additional coefficient identities;
  the $Q$-system condition restricts the general reconstruction problem.\\
\hline
Evans--Gannon \cite{EG2011}
& Solutions of Izumi's restricted equations and modular data of their doubles.
& Exact small-order results and numerical evidence at larger odd cyclic
  orders concern this $Q$-system family.\\
\hline
\raggedright Grossman--Snyder \cite{GS}
& \raggedright The third category $\mathcal H_3$ in the Haagerup Morita class.
& At $G=\Z/3\Z$, a distinct tensor category with the same fusion
  rules as the six-object even part.\\
\hline
Evans--Gannon \cite{EG}
& Leavitt-algebra reconstruction allowing both dimension roots without
  the additional $Q$-system equations.
& Includes nonunitary analogues of the Haagerup and Grossman--Snyder
  examples.\\
\hline
Huang--Lin: transparent $A_4$ \cite{HL}
& Trivial associators whenever an object is invertible; this gauge
  implies orientation-preserving tetrahedral invariance.
& The full transparent ansatz considered there; includes the
  additional $S_4$ locus.\\
\hline
Huang--Lin: transparent $S_4$ \cite{HL}
& Further requires tetrahedral reflection invariance without complex
  conjugation.
& A restriction on transparent solutions.\\
\hline
\end{tabular}
\caption{A comparison of various systems proposed in references.}
\end{table}
Huang and Lin match their strict-$A_4$ solutions with Izumi systems,
and call the $S_4$ examples Grossman--Snyder systems
\cite[Section~6]{HL}; their proposed correspondence for arbitrary $G$
is conjectural.

\subsection*{What was known}

A fusion ring is not a category. Existence for a given $G$ requires
solving the pentagon equations, and Izumi's endomorphism method turns
this into a finite system of polynomial equations in the coefficients
that define $\rho$ on the Cuntz generators. This approach has yielded
exact constructions and classifications for several specified groups.
Izumi established uniqueness for $\Z/3\Z$ and $\Z/5\Z$ within his
$Q$-system setting \cite{Izumi2}. Evans and Gannon found one subfactor
of Izumi type for $\Z/7\Z$, two for $\Z/9\Z$, and no solution of
those restricted equations for $\Z/3\Z\times\Z/3\Z$. They also
reported numerical evidence for cyclic orders $11,13,15,17,19$ and
proposed infinite families of twisted Haagerup--Izumi modular data
whose realizability was left open \cite{EG2011}. These are not
classification or nonexistence statements for all categories with the
corresponding fusion rules. Izumi's work on $3^n$ subfactors further
reduced the classification to polynomial equations under a
cohomological assumption automatic for odd-order groups, and produced
additional examples for even-order groups \cite{Izumi2018}.
Huang and Lin computed explicit $F$-symbols in a transparent gauge, in
which the associators involving an invertible object are trivial, and
classified transparent solutions for odd cyclic groups through order
nine, and through order fifteen under the additional $S_4$ condition
in the table \cite{HL}. On the structural side, Grossman and Snyder showed
that the Morita class of the Haagerup category contains exactly three
fusion categories \cite{GS}. Grossman and Izumi obtained general
partial formulas and explicit examples of Drinfeld-center modular data
for generalized Haagerup categories with $Q$-systems \cite{GI}.
Grossman, Izumi and Snyder related the Asaeda--Haagerup even parts by
higher Morita equivalence to an orbifold quotient of a generalized
Haagerup category \cite{GIS}, and classified certain
$\Z/2\Z$-graded extensions of generalized Haagerup categories
\cite{GIS2}.

Arithmetic constraints are also well developed. Morrison and Snyder
proved that a Haagerup even part cannot be defined over a cyclotomic
field \cite{MS}. Ostrik's Galois conditions on formal codegrees
\cite{OstrikFC} give necessary conditions for categorification, and
Zheng, Bao and Yu computed the Casimir numbers, irreducible ring
representations, and formal codegrees of the cyclic Haagerup--Izumi rings
for all cyclic orders; in particular, their formal codegrees satisfy
the pseudo-unitary inequality \cite{ZBY}. Meanwhile the Haagerup categories
have become test cases in physics, through lattice models and numerical
evidence for a Haagerup conformal field theory \cite{HLOTT,Vanhove},
and through an explicit construction of a gapless phase with Haagerup
symmetry \cite{Bottini}.

\subsection*{Nonunitary reconstruction}

The route taken here starts from the nonunitary side. Evans and Gannon
replaced the $C^*$-algebraic framework by endomorphisms of Leavitt
algebras and proved a reconstruction theorem \cite[Theorem~2(a)]{EG}:
for a finite abelian group $G$ of odd order, a matrix $\mathsf A$
that satisfies their Eqs.~(4.7)--(4.10), with their parameter
$\omega=1$ and a choice of root $\delta$ of
$\delta^2=n\delta+1$, yields a spherical fusion category with
Haagerup--Izumi fusion rules and categorical dimension $\delta$ for the
noninvertible simples. These are linear, quadratic and quartic
relations in the entries of $\mathsf A$. Cubic identities enter our
proof as an intermediate step toward the quartic identity. The
negative root $\delta=-1/d$ is allowed, and the coefficients constructed
below have a natural analytic description in this normalization.

\subsection*{The construction}

Fix odd $n\ge3$ and take the hyperbolic gamma function
$\gamma^{(2)}(z;\omega_1,\omega_2)$ in the normalization of
Belousov, Sarkissian and Spiridonov \cite{BSS}. Choose the periods
\[
 \omega_1=d+1,\qquad \omega_2=1+d^{-1},\qquad \omega_1-\omega_2=n .
\]
Sampling $F(z)=1/\gamma^{(2)}(\omega_2+z)$ at the integers, with the
signs and normalization in \eqref{foundation:coefficients}, gives
nonzero real numbers $W_a$ with $W_0=-1$ and
$W_aW_{-a}=-1/d$ for $a\ne0$. Form
$\mathsf B_{a,b}=W_aW_{b-a}/W_b$ for $b\ne0$, with the modified zeroth
column in \eqref{foundation:core}. The reconstruction matrix is
$\mathsf A=\tau^{-1}\mathsf B$, where $\tau=-1-d^{-1}$.
The integer period difference makes the gamma shift identities
compatible with indexing by $\Z/n\Z$. Together with
$\Omega:=\omega_1+\omega_2=\omega_1\omega_2$, the reflection identity
and the sample signs yield the reciprocal products needed for
$\delta=-1/d$.

The proof that this matrix satisfies the Evans--Gannon identities has
four steps, described in Section~\ref{sec:overview}. The quadratic
identities follow from a finite Fourier transform of each column and
from reflection. The cubic identities are the heart of the matter.
We show that a range of their Fourier coefficients vanishes by a
four-term contour relation proved below from the gamma shift identities.
The analytic inputs include the gamma reflection and shift identities
and the two-gamma integral evaluation of \cite{BSS}. The remaining cubics
are recovered algebraically, through permutation symmetries, polynomial
interpolation, and matrix inversion. The quartics follow by contraction.
Applying reconstruction with $\delta=-1/d$ gives a complex spherical
category with negative noninvertible dimension, the analogue of
Yang--Lee. Extending the nontrivial automorphism of $\Q(\sqrt{n^2+4})$
to an automorphism of $\C$ and reconstructing again gives a category with
dimensions $1$ and $d$, which is therefore pseudo-unitary. This is
Theorem~\ref{intro:main}.

\medskip
\noindent\textbf{Note added.}
T. Gannon has informed the author that joint work with A. Schopieray
and H. Yadav \cite{GSY1}, currently in preparation, will contain results
partially overlapping with those of this paper.

%% file: proof_revised/analytic_setup.tex
\section{The coefficients and the quadratic identities}

The construction has two parts: explicit real coefficients define the
matrix $\mathsf B$, and a finite Fourier transform proves its quadratic
identities.
Fix an odd integer $n\geq3$, put $G=\Z/n\Z$, and set
\[
 d=\frac{n+\sqrt{n^2+4}}2,\qquad \delta=-d^{-1},\qquad
 \tau=\delta-1,\qquad
 \omega_1=d+1,\qquad \omega_2=1+d^{-1}.
\]
The relations used below are
\[
 \delta^2=n\delta+1,\qquad \omega_1-\omega_2=n,\qquad
 \Omega:=\omega_1+\omega_2=\omega_1\omega_2,\qquad
 \omega_1/\omega_2=d,\qquad \tau=-\omega_2.
\]
Since $n^2<n^2+4<(n+1)^2$, the numbers $d$ and $\omega_2$ are
irrational. In particular,
$(\omega_1\Z+\omega_2\Z)\cap\Z=n\Z$: writing
$r\omega_1+s\omega_2=rn+(r+s)\omega_2$ shows that an integer value
requires $r+s=0$.

\subsection{Double-sine samples and the coefficient matrix}\label{coeff}

Let $\gamma(z)=\gamma^{(2)}(z;\omega_1,\omega_2)$ be the hyperbolic
gamma function in the normalization of \cite{BSS}. On
$0<\Re z<\Omega$ it is defined by the absolutely convergent integral
\[
 \gamma(z)=\exp\left\{-\int_0^\infty
 \left[\frac{\sinh((2z-\Omega)t)}
 {2\sinh(\omega_1t)\sinh(\omega_2t)}
 -\frac{2z-\Omega}{2\omega_1\omega_2t}\right]\frac{dt}{t}\right\},
\]
and is continued by its shift identities. The subtraction is kept
inside the integral to cancel the singularity at zero. The identities
we use, understood as identities of meromorphic functions, are
\begin{align}
 \gamma(z+\omega_1)&=2\sin(\pi z/\omega_2)\gamma(z),&
 \gamma(z+\omega_2)&=2\sin(\pi z/\omega_1)\gamma(z),
 \label{foundation:shifts}\\
 \gamma(z)\gamma(\Omega-z)&=1,&
 \Res_{z=0}\gamma(z)&=\frac{\sqrt\Omega}{2\pi}.
 \label{foundation:reflection}
\end{align}
Its poles are $-r\omega_1-s\omega_2$ and its zeros are
$\Omega+r\omega_1+s\omega_2$, for integers $r,s\geq0$.
They are simple because $\omega_1/\omega_2=d$ is irrational.
The integral gives positivity on $(0,\Omega)$, and the shifts give
reality on the real axis away from poles. Taking the first shift
identity at $z=0$ by a limit gives
$\gamma(\omega_1)=\sqrt\Omega/\omega_2=\sqrt d$; reflection gives
$\gamma(\omega_2)=1/\sqrt d$. Thus, for
$F(z)=1/\gamma(\omega_2+z)$,
\[
 F(0)=\sqrt d,\qquad F(z)F(n-z)=1,\qquad
 F(-a)F(a)=(-1)^{a+1}\quad(a\in\Z\setminus\{0\}).
\]
We have the following integer shift relation:
\begin{equation}\label{foundation:shift}
  F(a+n) = (-1)^{a+1}F(a)\quad(a\in\Z\setminus\{0\}).
\end{equation}
The excluded sample has the separate value $F(n)=1/\sqrt d$.
The poles of $F$ start at $\omega_1>n$ and its zeros at
$-\omega_2<-1$; hence the samples used to define the coefficients
are finite and nonzero.

Define the nonzero real coefficients
\begin{equation}\label{foundation:coefficients}
 W_0=-1,\qquad
 W_a=-(-1)^{a(a+n)/2}\frac{F(a)}{\sqrt d}
 \quad(1\leq a<n).
\end{equation}
Here the displayed integers specify the representatives of $G$.
In particular,
\begin{equation}\label{foundation:reciprocal-products}
 W_aW_{-a}=\delta\qquad(a\in G\setminus\{0\}).
\end{equation}
The matrix and its rescaling are
\begin{equation}\label{foundation:core}
 \mathsf B_{a,b}=
 \begin{cases}
  W_aW_{b-a}/W_b,&b\ne0,\\
  -1+\tau\ind{a=0},&b=0,
 \end{cases}
 \qquad \mathsf A_{a,b}=\tau^{-1}\mathsf B_{a,b}.
\end{equation}
Thus $\mathsf B_{0,0}=\delta-2$, and every other entry on either
axis or the diagonal equals $-1$. The reciprocal product identity
\eqref{foundation:reciprocal-products} also gives the order-three symmetry
\begin{equation}\label{foundation:rotation}
 \mathsf B_{a,b}=\mathsf B_{b-a,-a}.
\end{equation}

\emph{Bounds and contour values.} Since we will be working with various contour integrals containing multiple copies of hyperbolic gamma functions, we discuss the general ideas here and mention some subtleties involved.
The vertical asymptotics of \cite[Section~1]{BSS} give
\[
 \log|\gamma(x+it)|
 =\frac{\pi(x-\Omega/2)}{\omega_1\omega_2}|t|+O(1).
\]
In particular, for fixed real $c$,
\[
 \left|\frac{\gamma(x+it+c)}{\gamma(x+it)}\right|
 \leq C e^{\pi c|t|/(\omega_1\omega_2)}
 \quad(|t|\geq T).
\]
Both estimates are uniform for $x$ in a bounded interval.
Here $\omega_1\omega_2=\Omega$. The first estimate
also supplies bounds locally uniform in complex shifts: their real
and imaginary parts remain bounded on compact parameter sets.
Consequently a strict exponential decay margin for any gamma product
below persists in a parameter neighbourhood.

For a parameter-dependent integral with designated left and right pole
families, we work on open domains where opposite families do not meet
and the infinite tails have a locally uniform bound $Ce^{-c|\Im z|}$,
$c>0$. The pole families below are finite unions of translates of
one-sided period lattices. Locally one can keep a pole-free separating
contour fixed, agreeing with a vertical line outside a compact set.
The gamma factors are jointly holomorphic away from their poles;
compactness bounds the finite part of this contour and the exponential
estimate bounds its tails. Cauchy's estimates then justify
differentiation under the integral, giving holomorphy in the parameter.
Relative to an upward pole-free line, the separating
integral is the line integral plus $2\pi i$ times the residues of
left poles to its right, minus $2\pi i$ times the residues of right
poles to its left. These sums are finite. With the measure $dz/i$,
the correction factors are $2\pi$ instead.
Colliding poles of the same family are enclosed together by a fixed
loop; its integral remains holomorphic and represents the full sum
of their residues, including the residue at a merged pole. This
description also proves local holomorphy of the corrected line value.
If the chosen line meets a pole, move it and adjust the residue sums.

\subsection{The Fourier transform of a column of the coefficient matrix}
\begin{lemma}[Fourier transform]\label{foundation:column-fourier}
Write $u_p(a)=\mathsf B_{a,p}$. For $1\leq h<n$ and $k\in G$, define
\[
 \widehat u_h(k)=\sum_{a\in G}u_h(a)e^{-2\pi ika/n},
 \qquad
 \theta\equiv\frac{2\pi k}{n}+\pi(h-1)\pmod{2\pi},
 \quad -\pi\leq\theta\leq\pi.
\]
Then
\begin{equation}\label{fd:fourier}
 \widehat u_h(k)
 =-\omega_2 e^{-ih\theta/2}
   \gamma\left(\frac{\Omega+h}{2}-\frac{\Omega\theta}{2\pi}\right)
   \gamma\left(\frac{\Omega+h}{2}+\frac{\Omega\theta}{2\pi}\right).
\end{equation}
\end{lemma}

\begin{proof}
We choose a kernel whose integer residues give the summands of
$\widehat u_h(k)$, up to a common normalization and one correction
at $a=h$. Set
\[
 R_h(z)=\frac{F(z)}{F(z-h)},\qquad
 K_h(z)=\frac{\pi e^{-i\theta z}R_h(z)}{\sin\pi z}.
\]
At the integer samples, including $a=0,h$, the coefficient formula gives
\begin{equation}\label{foundation:sample}
 u_h(a)+\omega_2\ind{a=h}
 =\frac{(-1)^{h(a+1)}}{\sqrt d\,F(h)}R_h(a)
 \qquad(0\leq a<n).
\end{equation}
The ratio $R_h$ is holomorphic at these samples, and the residue of
$\pi/\sin\pi z$ at $z=a$ is $(-1)^a$. Consequently
\[
 \Res_{z=a}K_h(z)=(-1)^a e^{-i\theta a}R_h(a).
\]
Put $c_h=(-1)^h/(\sqrt d\,F(h))$. Since the choice of $\theta$ gives
$e^{-i\theta a}=(-1)^{a(h-1)}e^{-2\pi ika/n}$, equation
\eqref{foundation:sample} therefore gives
\[
 c_h\Res_{z=a}K_h(z)
 =\bigl(u_h(a)+\omega_2\ind{a=h}\bigr)e^{-2\pi ika/n}.
\]
Summing over the representatives $a=0,\ldots,n-1$ yields
\[
 \widehat u_h(k)
 =c_h\sum_{a=0}^{n-1}\Res_{z=a}K_h(z)
       -\omega_2e^{-2\pi ikh/n}.
\]
It remains to evaluate this integer residue sum.

\emph{The difference of the kernels.}
Put $\alpha=1/\omega_1$ and $\beta=1/\omega_2$, so
$\alpha+\beta=1$. Applying the shift identities and
\[
 (-1)^h\sin(\pi\alpha z)\sin(\pi\beta(z-h))
 -\sin(\pi\beta z)\sin(\pi\alpha(z-h))
 =\sin(\pi\alpha h)\sin\pi z
\]
gives
\begin{equation}\label{foundation:kernel-difference}
 K_h(z+n)-K_h(z)
 =-4\pi\sin(\pi h/\omega_1)e^{-i\theta z}\gamma(-z)\gamma(z-h).
\end{equation}

\emph{The residue rectangle.}
Let $L=\lfloor h/\omega_2\rfloor$ and choose
$\sigma=-\varepsilon$, where
$0<\varepsilon<\min\{1,(L+1)\omega_2-h\}$.
In the strip $\sigma<\Re z<\sigma+n$ (Figure~\ref{fig:contour}),
the poles of $K_h$ are the integers $0,\ldots,n-1$ from the cosecant
and $p_\ell=h-\ell\omega_2$ for $1\leq\ell\leq L$ from the
denominator of $R_h$.
The first pole of $F$ is at $\omega_1>n$, so the numerator of $R_h$ contributes
none. Irrationality of $\omega_2$ excludes coincident poles, and
$K_h(z+n)$ is regular at the additional poles $p_\ell$.
Indeed $p_\ell+n$ is not an integer. A numerator pole there would put
$h$ in $(\omega_1\Z+\omega_2\Z)\cap\Z=n\Z$, contrary to $0<h<n$.
A denominator zero would give
$n-\ell\omega_2=-r\omega_1-(s+1)\omega_2$, or
$(r+1)n=(\ell-r-s-1)\omega_2$, for $r,s\geq0$; irrationality makes
this impossible.
\begin{figure}
  \centering
  \includegraphics[width=0.7\linewidth]{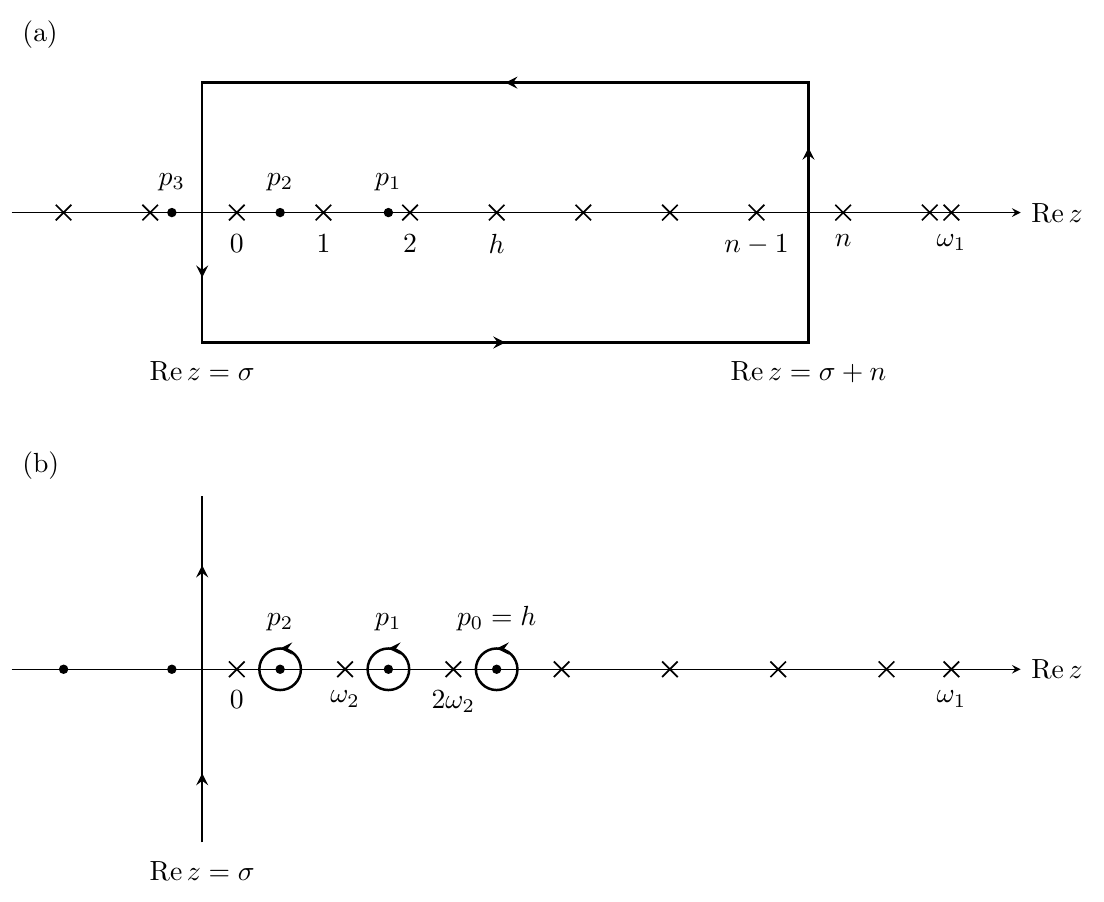}
  \caption{Contours for \eqref{foundation:rectangle}: (a) the residue rectangle, whose vertical sides are compared using \eqref{foundation:kernel-difference}; (b) the contour for $I_{\rm sep}$.}
  \label{fig:contour}
\end{figure}

The vertical ratio bound gives
\[
 |R_h(x+it)|=O(e^{-\pi h|t|/\Omega})
\]
uniformly on the closed strip. Since
$|\sin\pi(x+it)|^{-1}=O(e^{-\pi|t|})$ and
$|e^{-i\theta(x+it)}|\leq e^{\pi|t|}$, the same decay bound holds
for $K_h$. Hence its horizontal integrals vanish and its vertical
integrals converge, even at $|\theta|=\pi$.

Write $p_0=h$ and set
\[
 r_\ell=\Res_{z=p_\ell}
       e^{-i\theta z}\gamma(-z)\gamma(z-h),\qquad
 I_\sigma=\int_{\sigma-i\infty}^{\sigma+i\infty}
       e^{-i\theta z}\gamma(-z)\gamma(z-h)\frac{dz}{i}.
\]
Reflection and the ratio bound give
\[
 |\gamma(-x-it)\gamma(x+it-h)|
 =\left|\frac{\gamma(x+it-h)}{\gamma(x+it+\Omega)}\right|
 =O(e^{-\pi(1+h/\Omega)|t|}).
\]
Thus $I_\sigma$ converges uniformly for $|\theta|\leq\pi$.
The pole $p_0=h$ of the two-gamma integrand is already among the
integer poles of $K_h$; $R_h$ itself is regular there.
The residue theorem and \eqref{foundation:kernel-difference} yield
\begin{equation}\label{foundation:rectangle}
 \sum_{a=0}^{n-1}\Res_{z=a}K_h(z)
 =-2\sin(\pi h/\omega_1)I_\sigma
  -4\pi\sin(\pi h/\omega_1)\sum_{\ell=1}^{L}r_\ell.
\end{equation}

\emph{The integral evaluation and cancellation.}
Choose an upward separating contour, agreeing with the line
$\Re z=\sigma$ outside a compact set, which puts the poles of
$\gamma(z-h)$ on its left and those of $\gamma(-z)$ on its right
(Figure~\ref{fig:contoursep}). Write $I_{\rm sep}$ for its integral
with integrand $e^{-i\theta z}\gamma(-z)\gamma(z-h)$ and measure
$dz/i$. The two-gamma Fourier integral \cite[Eq.~(2.17)]{BSS}, with shifts
$-h,0$ and parameter
$\alpha_{\rm BSS}=(\Omega+h-\Omega\theta/\pi)/2$, gives
\begin{align}
 I_{\rm sep}
 &=\sqrt\Omega\,e^{-ih\theta/2}\gamma(-h)
   \gamma\left(\frac{\Omega+h}{2}-\frac{\Omega\theta}{2\pi}\right)
   \gamma\left(\frac{\Omega+h}{2}+\frac{\Omega\theta}{2\pi}\right),
   \label{foundation:beta}\\
 I_{\rm sep}&=I_\sigma+2\pi\sum_{\ell=0}^{L}r_\ell.
   \label{foundation:continued-contour}
\end{align}
\begin{figure}
  \centering
  \includegraphics[width=0.7\linewidth]{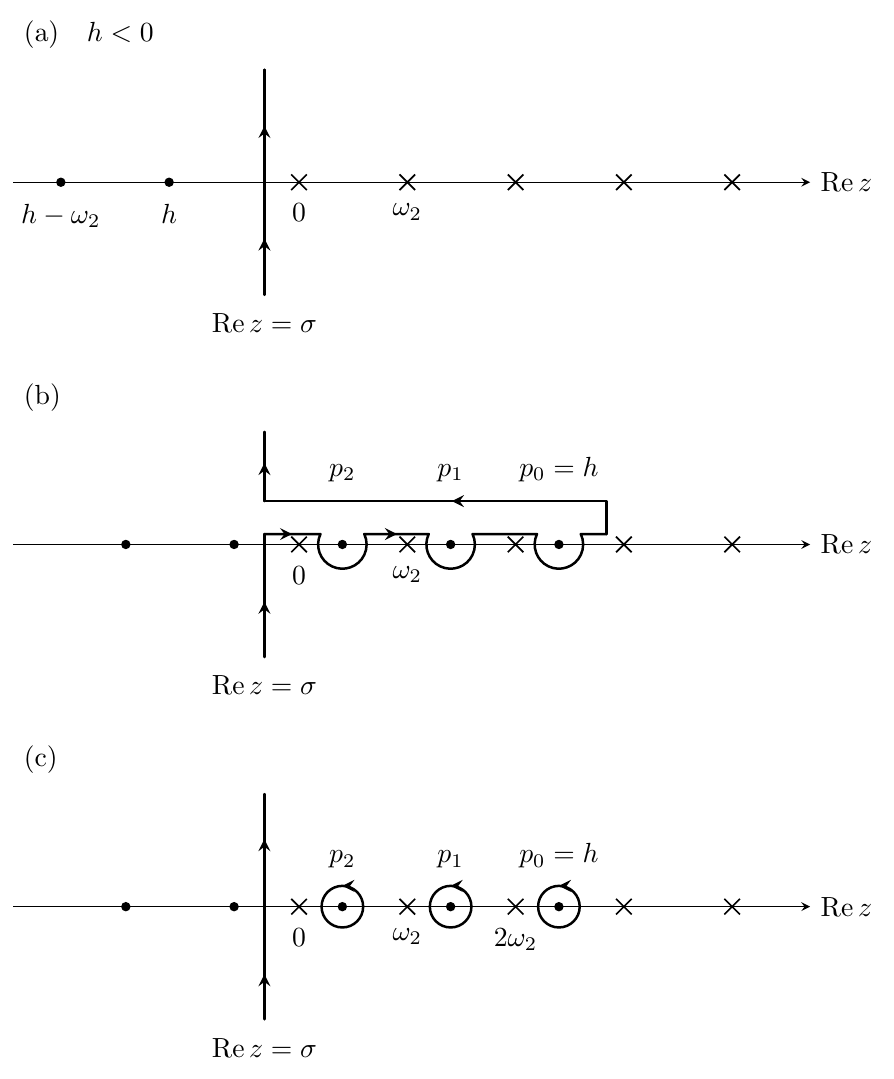}
  \caption{The separating contour as we vary $h$.}
  \label{fig:contoursep}
\end{figure}
To justify the continuation, first fix $|\theta|<\pi$ and replace
$h$ by a complex parameter $h'$. Its domain is the connected open set
\[
 U_\theta=\{h'\in\C:\Re h'>\Omega(|\theta|/\pi-1)\}
 \setminus(\omega_1\Z_{\geq0}+\omega_2\Z_{\geq0}).
\]
The inequality is exactly the strict decay condition for the
two-gamma integrand; the deleted discrete set is exactly the set of
opposite-pole collisions. Choose a real
$h_0\in(\max\{-1,\Omega(|\theta|/\pi-1)\},0)$ and a line
$h_0<\sigma_0<0$, independently of the final line $\sigma$.
Near $h_0$ this line separates the families and the BSS formula
applies directly. The contour prescription above gives a holomorphic
integral on $U_\theta$; its displayed gamma evaluation is also
holomorphic there, since $-h'$ avoids the pole lattice and the other
two arguments have positive real parts. The identity theorem proves
\eqref{foundation:beta} throughout $U_\theta$, in particular at
$h$: an integer $0<h<n$ cannot lie in the period lattice.
At this parameter the left poles to the right of $\sigma$ are exactly
$p_0,\ldots,p_L$, and there are no crossed right poles. The residue
prescription therefore gives \eqref{foundation:continued-contour}.
Finally let $\theta\to\pm\pi$ at fixed $h>0$. The bound for
$I_\sigma$ is uniform, the finitely many residue terms are continuous,
and the gamma arguments at the endpoints, $h/2$ and $\Omega+h/2$,
are positive. This proves both formulas at the endpoint frequencies.

Substituting \eqref{foundation:continued-contour} into
\eqref{foundation:rectangle} cancels all residues $r_\ell$ with
$\ell\geq1$ and gives
\[
 \sum_{a=0}^{n-1}\Res_{z=a}K_h(z)
 =-2\sin(\pi h/\omega_1)I_{\rm sep}
       +4\pi\sin(\pi h/\omega_1)r_0.
\]
The remaining residue and normalization are
\[
 r_0=\frac{\sqrt\Omega}{2\pi}e^{-i\theta h}\gamma(-h),\qquad
 2\sin(\pi h/\omega_1)\gamma(-h)=(-1)^hF(h),\qquad
 \frac{\sqrt\Omega}{\sqrt d}=\omega_2.
\]
Returning to the residue formula for $\widehat u_h(k)$ at the start
of the proof, and using \eqref{foundation:beta}, we obtain
\begin{align*}
 \widehat u_h(k)
 &=-2c_h\sin(\pi h/\omega_1)I_{\rm sep}
     +4\pi c_h\sin(\pi h/\omega_1)r_0-\omega_2e^{-2\pi ikh/n}\\
 &=-\omega_2e^{-ih\theta/2}
   \gamma\left(\frac{\Omega+h}{2}-\frac{\Omega\theta}{2\pi}\right)
   \gamma\left(\frac{\Omega+h}{2}+\frac{\Omega\theta}{2\pi}\right)\\
 &\qquad+\omega_2\bigl(e^{-i\theta h}-e^{-2\pi ikh/n}\bigr).
\end{align*}
Since $h(h-1)$ is even, the frequency congruence implies
$e^{-i\theta h}=e^{-2\pi ikh/n}$. Thus the $r_0$ contribution cancels
the diagonal sample correction, leaving exactly the stated Fourier
transform \eqref{fd:fourier}.
\end{proof}

\subsection{Fourier inversion gives the quadratic identities}

\begin{lemma}[Reciprocal Fourier products]\label{foundation:fourier-products}
For $h\in G\setminus\{0\}$ and $k\in G$,
\begin{equation}\label{foundation:spectral-inverse}
 \widehat u_h(k)\widehat u_{n-h}(-k)
 =\omega_2^2e^{-2\pi ikh/n}.
\end{equation}
\end{lemma}

\begin{proof}
This follows from \eqref{fd:fourier} by reflection and one period
shift; we include the sign calculation. Swapping $h,n-h$ and taking
complex conjugates reduces to odd $h$ and
$0\leq k\leq(n-1)/2$. With $\lambda_k=\Omega k/n$, the four gamma
arguments are
\[
 x_\pm=\frac{\Omega+h}{2}\pm\lambda_k,\qquad
 y_1=\Omega+\frac{n-h}{2}-\lambda_k,\qquad
 y_2=\frac{n-h}{2}+\lambda_k.
\]
Since $x_++y_1=\Omega+\omega_1$ and $x_-+y_2=\omega_1$,
their product reduces to
\[
 \frac{\sin\!\left(\pi[(h-n)/(2\omega_2)+\omega_1k/n]\right)}
      {\sin\!\left(\pi[(\Omega+h)/(2\omega_1)-\omega_2k/n]\right)}
 =(-1)^{(h-n)/2+k}.
\]
The two sine arguments divided by $\pi$ sum to the integer
$(h-n+2)/2+k$. The denominator argument divided by $\pi$ is
\[
 B=\frac{h}{2\omega_1}+\omega_2\left(\frac12-\frac{k}{n}\right)>0,
 \qquad B\leq\frac{\Omega+h}{2\omega_1}<1,
\]
so the division is legitimate. Combining the resulting sign with the
exponential factors
proves \eqref{foundation:spectral-inverse}.
\end{proof}

\begin{proposition}[Quadratic identities]\label{foundation:quadratics}
The matrix \eqref{foundation:core} satisfies
\begin{equation}\label{foundation:Q}
 \sum_{a\in G}\mathsf A_{g+a,h}\mathsf A_{h,a}
 =\ind{g=0}-\delta^{-1}\ind{h=0}
 \qquad(g,h\in G).
\end{equation}
\end{proposition}

\begin{proof}
For $h\ne0$, the symmetry \eqref{foundation:rotation} gives
$\mathsf B_{h,a}=\mathsf B_{a-h,-h}$.
The Fourier transform in $g$ of the unscaled left side is therefore
\[
 \widehat u_h(k)e^{2\pi ikh/n}
 \widehat u_{-h}(-k)=\omega_2^2=\tau^2.
\]
Fourier inversion gives $\tau^2\ind{g=0}$, as required.
For $h=0$, write $\mathsf B_{a,0}=\mathsf B_{0,a}
=-1+\tau\ind{a=0}$. The sum is
\[
 n-2\tau+\tau^2\ind{g=0}
 =\tau^2\bigl(\ind{g=0}-\delta^{-1}\bigr),
\]
where the last equality uses $\delta^2=n\delta+1$.
Dividing by $\tau^2$ proves the claim.
\end{proof}

%% file: proof_revised/equations.tex
\section{The remaining equations}
\label{sec:targets}

We now state the algebraic target for the contour calculation.
The coefficient formula \eqref{foundation:core} can also be written
\begin{equation}\label{tail:uniform-core}
 \mathsf B_{a,b}=\delta^{-1}W_aW_{b-a}W_{-b}
                  +\tau^2\delta^{-1}\ind{a=b=0}.
\end{equation}
It immediately gives the linear reconstruction identities
\begin{equation}\label{tail:linear}
 \mathsf A_{a,b}=\mathsf A_{-b,a-b},\qquad
 \mathsf A_{a,0}=\mathsf A_{0,a}=\ind{a=0}-\tau^{-1},
 \qquad \sum_a\mathsf A_{a,0}=-\delta^{-1}.
\end{equation}
Define the residuals
\begin{align}
 Q(g,h)&=\sum_m\mathsf A_{m+g,h}\mathsf A_{h,m}
                  -\ind{g=0}+\delta^{-1}\ind{h=0},\label{tail:Q}\\
 C(g,h,k,l)&=\sum_m\mathsf A_{m,g+h}\mathsf A_{g,m+k}
                         \mathsf A_{h,m+l}
       -\mathsf A_{g+l,k}\mathsf A_{h+k,l}
       +\delta^{-1}\ind{g=h=0},\label{tail:C}\\
 \mathcal C_{g,h}(k,l)&=\tau^3C(g,h,k,l).\label{tail:scaledC}
\end{align}
Proposition~\ref{foundation:quadratics} says $Q=0$.
Parts~\ref{part:contours} and~\ref{part:completion} prove $C=0$;
Part~\ref{part:reconstruction} then derives the quartic equations.
The two normalizations $C$ and $\mathcal C$ express the same condition:
$C$ uses the reconstruction matrix $\mathsf A$, while $\mathcal C$
avoids denominators in calculations with $\mathsf B$.

%% file: proof_revised/contour_cubics.tex
\section{Fourier reduction of the cubic equations}
\label{an:products}

We first prove that many Fourier coefficients of the cubic residual
vanish.  Finite-dimensional algebra will then recover the remaining
coefficients.  Write $u_p(a)=\mathsf B_{a,p}$,
with indices in $G$, and set
\[
 \widehat u_p(k)=\sum_{a\in G}u_p(a)e^{-2\pi ika/n},\qquad
 S_{p,q,s}(k)=\sum_{a\in G}u_p(a)u_q(a+s)e^{-2\pi ika/n}.
\]
\subsection{Finite Fourier reduction}

For the residual \eqref{tail:scaledC} with $g,h,g+h\ne0$,
the order-three symmetry $\mathsf B_{g,a}=u_{-g}(a-g)$ gives
\begin{equation}
 \sum_k\mathcal C_{g,h}(k,l)e^{-2\pi ijk/n}
 =e^{-2\pi ijg/n}\left[
  \widehat u_{-g}(j)S_{g+h,-h,l-h}(-j)
       -\tau e^{-2\pi ijl/n}S_{-g-l,l,g+h+l}(j)\right].
\label{an:finite}
\end{equation}
Indeed, transform the first term in $k$ and substitute $a=k-g-l$
in the second.  Thus it suffices to relate two product transforms.

We use signed column labels.  For $-n<q<0$,
\begin{equation}
 u_q(a)=\frac{(-1)^{q(a+1)}}{\sqrt d\,F(q)}
                      \frac{F(a)}{F(a-q)},\qquad 0\le a<n.
 \label{an:signed}
\end{equation}
There is no sample correction: at $a=0$ the reciprocal product identity gives
$-1$, and at $a=n+q$ the identity $F(n)=1/\sqrt d$ gives $-1$.
The remaining samples follow from \eqref{foundation:coefficients} and the
integer sine shifts.

\subsection{A product transform with its sample corrections}

For arbitrary integer labels introduce
\begin{gather*}
 R_p(z)=\frac{F(z)}{F(z-p)},\qquad
 K(z)=\frac{\pi e^{-i\theta z}}{\sin\pi z}R_p(z)R_q(z+s),\\
 \theta\equiv\frac{2\pi k}{n}+\pi(p+q-1)\pmod{2\pi},\qquad
 \theta\in[-\pi,\pi],
\end{gather*}
and
\begin{align*}
 J_1(z)&=\gamma(-z)\gamma(z-p)
                   \gamma(\omega_2-z-s)\gamma(\omega_1+z+s-q),\\
 J_2(z)&=\gamma(\omega_1-z)\gamma(\omega_2+z-p)
                   \gamma(-z-s)\gamma(z+s-q).
\end{align*}
The sine-shift equations give the meromorphic identity
\begin{equation}
 K(z+n)-K(z)=-4\pi e^{-i\theta z}
 \left[(-1)^q\sin\frac{\pi p}{\omega_1}J_1(z)
       +(-1)^s\sin\frac{\pi q}{\omega_1}J_2(z)\right].
 \label{an:two-ratio}
\end{equation}
To see this, put $M_p=R_p(z+n)/R_p(z)$ and
$M_q=R_q(z+s+n)/R_q(z+s)$ and split the multiplier into
\[
 (-1)^{p+q}M_pM_q-1
 =\bigl((-1)^pM_p-1\bigr)(-1)^qM_q+
                                      \bigl((-1)^qM_q-1\bigr).
\]
Apply \eqref{foundation:kernel-difference}
to each bracket.  The remaining factors are
$R_q(z+s+n)=\gamma(\omega_2-z-s)\gamma(\omega_1+z+s-q)$ and
$R_p(z)=\gamma(\omega_1-z)\gamma(\omega_2+z-p)$; also
$\sin\pi(z+s)/\sin\pi z=(-1)^s$.

In the rest of this section assume
\begin{equation}
 g,h\ge1,\qquad p=g+h<n,\qquad 1\le l\le h,
 \qquad q=-h,\quad s=l-h.
 \label{an:ordered}
\end{equation}
For any labels define
\[
 c_0=\frac{(-1)^{p+q+qs}}{dF(p)F(q)},\qquad
 a_0=(-1)^q\sin\frac{\pi p}{\omega_1},\qquad
 b_0=(-1)^s\sin\frac{\pi q}{\omega_1},\qquad
 I_i=\int_{\mathcal C_i}e^{-i\theta z}J_i(z)
                                    \frac{dz}{i\sqrt\Omega},
\]
where $\mathcal C_i$ is the continued separating\footnote{An upward separating contour places the poles of each $\gamma(A_i+z)$ to its left and those of each $\gamma(B_j-z)$ to its right. A vertical line suffices when $\max_i(-\Re A_i)<\min_j\Re B_j$.}
contour for its gamma factors, with the residue prescription in the
gamma preliminaries.  Either endpoint representative is allowed when
$\theta\equiv\pi\pmod{2\pi}$: the integer residues and the weighted
combination in \eqref{an:product} agree, although the individual $I_i$
need not be periodic in $\theta$.

\begin{lemma}[Product transform]\label{an:product-formula}
For the labels in \eqref{an:ordered}, at every discrete frequency,
\begin{equation}
 S_{p,q,s}(k)=-2\sqrt\Omega\,c_0(a_0I_1+b_0I_2).
 \label{an:product}
\end{equation}
The same formula holds for
\[
 (p',q',s')=(-g-l,l,g+h+l)
\]
provided its principal frequency satisfies
\begin{equation}
 |\theta|<\pi(1-g/\Omega).
 \label{an:frequency}
\end{equation}
\end{lemma}

\begin{proof}
We first specify the continuation used below.  Write each $J_i$ as
$\gamma(A_1+z)\gamma(A_2+z)\gamma(B_1-z)\gamma(B_2-z)$.
Its total parameter sum is $\Sigma=\Omega-p-q$, equal to $\Omega-g$
for the first product and $\Omega+g$ for the second.  Deform all four
parameters $a_j$ by
\[
 a_j(\lambda)=(1-\lambda)\Sigma/4+\lambda a_j,\qquad \lambda\in\C.
\]
At $\lambda=0$ the imaginary axis separates the poles.  The total sum
stays fixed, so on bounded vertical strips the integrands have the
locally uniform tail bound
\[
 O\!\left(\exp\left\{
   \left(|\theta|-\frac{\pi(2\Omega-\Sigma)}{\Omega}\right)|\Im z|
                    \right\}\right).
\]
The exponent is negative in the stated frequency ranges.  Opposite
families can meet only when some $A_i(\lambda)+B_j(\lambda)$ belongs to
$-\omega_1\Z_{\ge0}-\omega_2\Z_{\ge0}$.  These exclude a locally
finite set of $\lambda$; a constant sum excludes none, since its value
at zero is $\Sigma/2>0$.  The complement is connected.  Once the
no-pinch checks below are made at $\lambda=1$, the fixed-line and
fixed-loop prescription therefore continues the initial integrals to
the required ones without leaving their convergence domain.

\emph{The first product.}
Let $L=\lfloor p/\omega_2\rfloor$ and choose
$\sigma=-\varepsilon$ with
$0<\varepsilon<\min(1,(L+1)\omega_2-p)$.
In $\sigma<\Re z<\sigma+n$ there are no numerator poles of $K$;
$F(z+s-q)=F(z+l)$ has no zero.  The noninteger poles are exactly
$z_j=p-j\omega_2$, $1\le j\le L$.  The shifted kernel $K(z+n)$ is
regular there.  Irrationality of $\omega_2$ excludes coincidences with
the other divisors.

To turn the vertical integrals into separating integrals one must
include the left poles $z_0=p,z_1,\ldots,z_L$ of $J_1$ and
$z_1,\ldots,z_L$ of $J_2$.  Their other left poles lie to the left
of $-l$ and their right poles are nonnegative.  The only possibly
negative opposite-parameter sums are $-p$ and $\omega_2-g-l$.
Neither is $-r\omega_1-t\omega_2$ for nonnegative integers $r,t$: since $\omega_1>n$,
such an equality would express a positive integer smaller than $n$
as a positive integer multiple of the irrational number $\omega_2$.
Thus the continued contours are nonsingular.

Write $r_{i,j}=\Res_{z=z_j}e^{-i\theta z}J_i(z)$.
Equation \eqref{an:two-ratio} implies, for $j\ge1$,
\[
 \Res_{z=z_j}K(z)=4\pi(a_0r_{1,j}+b_0r_{2,j}).
\]
The changes from vertical to separating integrals add respectively
\[
 \frac{2\pi}{\sqrt\Omega}\sum_{j=0}^Lr_{1,j}
 \qquad\text{and}\qquad
 \frac{2\pi}{\sqrt\Omega}\sum_{j=1}^Lr_{2,j}.
\]
The residue rectangle consequently gives
\begin{equation}
 \sum_{a=0}^{n-1}\Res_{z=a}K(z)
 =-2\sqrt\Omega(a_0I_1+b_0I_2)+4\pi a_0r_{1,0}.
 \label{an:rectangle}
\end{equation}
The horizontal boundaries vanish because
$R_p(z)R_q(z+s)=O(e^{-\pi g|\Im z|/\Omega})$ uniformly on the
finite strip and $|\theta|\le\pi$.

The negative column has no correction in this sample interval; the
positive column has only its diagonal correction.  Hence
\[
 S_{p,q,s}(k)=c_0\sum_a\Res_{z=a}K(z)
                    -\omega_2 u_q(p+s)e^{-2\pi ikp/n}.
\]
On the other hand,
\[
 r_{1,0}=\frac{\sqrt\Omega}{2\pi}e^{-i\theta p}
                 \gamma(-p)R_q(p+s+n),\qquad
 c_0\,4\pi a_0r_{1,0}=\omega_2 u_q(p+s)e^{-2\pi ikp/n}.
\]
Here one uses $2\sin(\pi p/\omega_1)\gamma(-p)=(-1)^pF(p)$ and,
at this integer sample, $R_q(p+s+n)=(-1)^qR_q(p+s)$.
This proves \eqref{an:product}.

\emph{The second product.}
Put $t_1=g+l$, $t_2=g+h$, $t=t_2+l$ and
sample the same periodic sum on $a=-t,\ldots,n-t-1$.
The raw columns have precisely the sample corrections $\omega_2$ at $a=-t_1$ and
$a=-t_2$, respectively.  This follows from the integer sine shifts
also when the interval contains negative multiples of $n$;
the numerator arguments avoid positive multiples of $n$.
Choose $\sigma=-t-\varepsilon$ with
\[
 0<\varepsilon<\min\{1,(\lfloor h/\omega_2\rfloor+1)\omega_2-h,
                         (\lfloor l/\omega_2\rfloor+1)\omega_2-l\}.
\]
There are no numerator poles.  The two noninteger pole families are
\[
 -t_1-j\omega_2\ (1\le j\le\lfloor h/\omega_2\rfloor),\qquad
 -t_2-j\omega_2\ (1\le j\le\lfloor l/\omega_2\rfloor).
\]
The shifted kernel is regular there.  The crossed left poles of
$J_1$ are the first family including $j=0$; those of $J_2$ are
the first family with $j\ge1$ and the second family including
$j=0$.  Its other left poles are left of $\sigma$, and all right
poles are right of $\sigma$.  All noninteger residue contributions
cancel by \eqref{an:two-ratio}.
Here the opposite-parameter sums are respectively
\[
 \{t_1,\omega_2-h,\omega_1+t_2,\Omega-l\},\qquad
 \{\Omega+t_1,\omega_2-h,\omega_1+t_2,-l\}.
\]
The only possibly negative entries, $\omega_2-h$ and $-l$, cannot
belong to $-\omega_1\Z_{\ge0}-\omega_2\Z_{\ge0}$: as before,
$\omega_1>n$ and irrationality of $\omega_2$ exclude this.  Thus the
second product also has no contour pinch.

If $h\ne l$, the two surviving integer contributions are
\[
 \omega_2 u_l(h)e^{2\pi ikt_1/n},\qquad
 \omega_2 u_{-t_1}(-t_2)e^{2\pi ikt_2/n},
\]
and cancel the two sample corrections separately.  If $h=l$, some noninteger
poles of $J_2$ and $K$ are double.  The same argument uses their
full residues, since \eqref{an:two-ratio} is a meromorphic identity.
All pole lists are unions of distinct locations: a merged pole is
counted once, with the residue of the full integrand.
At the common integer point the first contribution uses the actual
second column, with value $-1$, whereas the second contribution
uses the raw first column, with value $1/d$.  Their sum is
\[
 \omega_2(1/d-1)e^{2\pi ikt_2/n}=(d^{-2}-1)e^{2\pi ikt_2/n},
\]
exactly the difference between the raw product $d^{-2}$ and the
actual product $1$.  Finally $p'+q'=-g$, so the ratio product grows like
$O(e^{\pi g|\Im z|/\Omega})$, and the horizontal boundaries vanish
under exactly \eqref{an:frequency}.
\end{proof}

Write a principal frequency as $\theta=\pi v/n$, with
$v\equiv g+1\pmod2$, considered modulo $2n$.  Since
\[
 n(1-g/\Omega)=n-g+2g/\omega_1,\qquad 0<2g/\omega_1<2,
\]
condition \eqref{an:frequency} is equivalent to
\begin{equation}
 |v|\le n-g.
 \label{an:restricted-set}
\end{equation}
There are $n-g+1$ such frequencies.

\section{Two identities for four-gamma integrals}
\label{an:euler-section}

The product formula leaves four-gamma integrals.  We use an Euler
transformation to put them in the same form, then a contour difference
identity to cancel their linear combination.  Define
\begin{gather*}
 U=\Omega/4+(r+s)/2,\qquad V=\Omega/4+(r-s)/2,\\
 K_4(r,s,x,y)=\int e^{2\pi iyz/\Omega}
       \gamma(U+z)\gamma(U-z)
       \gamma(V+z+x)\gamma(V-z-x)\frac{dz}{i\sqrt\Omega},\\
 \Gamma_r(v)=\gamma(\Omega/2+r+v)\gamma(\Omega/2+r-v).
\end{gather*}
We use these integrals on the convergent, nonpinched domain
\[
 |\Re y|<\Omega/2-\Re r,\qquad
 2U,\ 2V,\ U+V+x,\ U+V-x
       \notin-\omega_1\Z_{\ge0}-\omega_2\Z_{\ge0}.
\]
The general vertical asymptotics of $\gamma$, also for complex shifts,
give the bound
\[
 \left|\gamma(U+z)\gamma(U-z)
          \gamma(V+z+x)\gamma(V-z-x)\right|
 =O\!\left(e^{\pi(2\Re r-\Omega)|\Im z|/\Omega}\right)
\]
on bounded vertical strips, locally uniformly in the parameters.
The contour prescription in the gamma preliminaries consequently
defines a holomorphic function on this domain.  Poles from the same
family may merge; a fixed loop then encloses their full residue sum.

\subsection{Euler transformation}

\begin{lemma}[Euler transformation]\label{an:euler}
Suppose $|\Re r|+|\Re y|<\Omega/2$ and neither of the two
four-gamma integrals below has a contour pinch.  Then
\begin{equation}
 K_4(r,s,x,y)=\frac{\Gamma_r(s)\Gamma_r(x)}{\Gamma_r(y)}K_4(-r,s,x,y).
 \label{an:euler-centered}
\end{equation}
\end{lemma}

\begin{proof}
Work first in $|r|+|s|+|x|+|y|<\Omega/16$.  All contours used below
can then be the imaginary axis.  With the normalization
\[
 \mathcal F\phi(y)=\int_{i\R}e^{2\pi iyz/\Omega}\phi(z)
                                     \frac{dz}{i\sqrt\Omega},
 \qquad
 \mathcal F\phi(i\eta)=\frac1{\sqrt\Omega}
       \int_{\R}e^{-2\pi it\eta/\Omega}\phi(it)\,dt,
\]
the product-convolution measure is $dy'/(i\sqrt\Omega)$.
The two-gamma Fourier integral \cite[Eq.~(2.17)]{BSS}, applied to
the two pairs with this normalization, gives
\begin{equation}
 K_4(r,s,x,y)=\Gamma_r(s)e^{-2\pi ixy/\Omega}K_4(-r,-s,-y,x).
 \label{an:D}
\end{equation}
Here is an absolute-convergence justification of the convolution.
Put $g_q(t)=\gamma(q+it)\gamma(q-it)$ and $H=\Omega/2-V$.
The two-gamma formula at the complementary parameter $H$, together
with reflection, gives directly
\[
 \gamma(V+x+it)\gamma(V-x-it)
 =\frac{\gamma(2V)}{\sqrt\Omega}
   \int_{\R}e^{2\pi xu/\Omega}e^{2\pi itu/\Omega}g_H(u)\,du.
\]
After insertion into $K_4$, the absolute value of the double integrand,
apart from its constant factor, is
\[
 \left|e^{-2\pi yt/\Omega}g_U(t)\right|
 \left|e^{2\pi xu/\Omega}g_H(u)\right|.
\]
Both factors are integrable: the neighborhood above satisfies
$0<\Re U,\Re V<\Omega/2$, $|\Re x|<\Re V$, and
$|\Re y|<\Omega/2-\Re U$.  Fubini and a second use of the two-gamma
formula yield $\Gamma_r(s)K_4(-r,s,y,x)$.
Exchanging the two gamma pairs in this last integral by translation
gives \eqref{an:D}, including its exponential factor.
Interchanging the two positive gamma parameters and translating $z$
by $(s-x)/2$ also gives
\begin{equation}
 K_4(r,s,x,y)=e^{\pi iy(s-x)/\Omega}K_4(r,x,s,y).
 \label{an:P}
\end{equation}
This substitution preserves the labelled pole families, so
\eqref{an:P} holds throughout its convergent, nonpinched domain.
Apply \eqref{an:D} and \eqref{an:P} alternately three times:
\begin{align*}
 (r,s,x,y)&\longmapsto(-r,-s,-y,x)
 \longmapsto(-r,-y,-s,x)\\
 &\longmapsto(r,y,-x,-s)
 \longmapsto(r,-x,y,-s)\\
 &\longmapsto(-r,x,s,y)
 \longmapsto(-r,s,x,y).
\end{align*}
The exponential factors cancel, since their combined exponent,
divided by $\pi i/\Omega$, is
\[
 -2xy+x(y-s)+2sx-s(y+x)+2sy+y(x-s)=0.
\]
The gamma factors multiply to
$\Gamma_r(s)\Gamma_r(x)/\Gamma_r(y)$, since reflection gives
$\Gamma_{-r}(y)=1/\Gamma_r(y)$.  This proves
\eqref{an:euler-centered} in the initial neighborhood.

For the stated parameters, scale $(r,s,x,y)$ by a complex variable
$\lambda$.  The common convergence domain
\[
 \{\lambda\in\C:
       |\Re(\lambda r)|+|\Re(\lambda y)|<\Omega/2\}
\]
is open and convex and contains $0$ and $1$.  Remove the locally
finite set where one of
$\Omega/2+\lambda(\pm r\pm s)$ or
$\Omega/2+\lambda(\pm r\pm x)$ is a negative period sum.
The remaining domain is connected and still contains $0$ and $1$.
Both contour integrals are holomorphic there.  The numerator gamma
factors have no poles there, and the denominator arguments
$\Omega/2+\lambda r\pm\lambda y$ have real parts in $(0,\Omega)$,
where $\gamma$ is finite and nonzero.  The one-variable identity
theorem therefore proves \eqref{an:euler-centered} at $\lambda=1$.
Only the final identity is continued; no convergence of the six
intermediate integrals is needed outside the initial neighborhood.
\end{proof}

\subsection{A contour difference identity}
\label{an:four-term-section}

Continue to assume \eqref{an:ordered}, and put
\[
 t_1=g+l,\qquad R=g/2,\qquad
 X=(n+g)/2+l,\qquad Y=(n+g)/2+h.
\]
\begin{lemma}[Four-term identity]\footnote{The bound on $Z$ is the common absolute-convergence condition for the four kernels and is supplied by the frequency restriction \eqref{an:restricted-set}.}\label{an:four-term}
If $|\Re Z|<(\Omega-g)/2$, then with continued separating contours
one has
\begin{align}
0={}&\sin(\pi h/\omega_1)K_4(R,X,Y+\omega_2,Z)\notag\\
 &+(-1)^{t_1+1}\sin(\pi p/\omega_1)e^{-2\pi iZ\omega_1/\Omega}
                                     K_4(R,X,Y-\omega_1,Z)\notag\\
 &+(-1)^p\sin(\pi t_1/\omega_1)e^{-\pi iZ\omega_1/\Omega}
                                     K_4(R,X-\omega_1,Y,Z)\notag\\
 &-\sin(\pi l/\omega_1)e^{-\pi iZ\omega_2/\Omega}K_4(R,X+\omega_2,Y,Z).
 \label{an:four-term-eq}
\end{align}
\end{lemma}

\begin{proof}
\emph{A meromorphic difference.}
Let $u=\Omega/4+(R+X)/2$ and
$v=\Omega/4+(R-X)/2$ be general parameters and define
\begin{gather*}
 \zeta=u-\omega_2+z,\quad \eta=u-\omega_1-z,\quad
 \xi=v-\omega_2+Y+z,\quad \nu=v-Y-z,\\
 c=2u-\Omega,\quad D=v+Y-u,\quad E=c+D,\quad
 f=u-v-\omega_1+Y.
\end{gather*}
Thus $\eta+\zeta=c$, $\eta+\xi=E$, and $\eta-\nu=f$.
In this proof abbreviate
\[
 a(t)=\sin(\pi t/\omega_1),\quad b(t)=\sin(\pi t/\omega_2),\quad
 \lambda=e^{2\pi iZ\omega_2/\Omega},
\]
and set
\[
 \kappa_1=\frac{a(f-c)}{a(D)},\quad
 \kappa_2=\lambda\frac{b(E)}{b(D)},\quad
 \kappa_3=-\lambda\frac{b(c)}{b(D)},\quad
 \kappa_4=\frac{a(E-f)}{a(D)}.
\]
Initially assume $D\notin\omega_1\Z\cup\omega_2\Z$, so all four
coefficients are defined.
Let $T_i(z)$ be $e^{2\pi iZz/\Omega}$ times the product in row $i$:
\[
\begin{array}{c|l}
 i&\text{gamma product}\\\hline
 1&\gamma(\zeta)\gamma(\eta+\Omega)\gamma(\xi+\omega_2)\gamma(\nu)\\
 2&\gamma(\zeta+\Omega)\gamma(\eta)\gamma(\xi+\omega_2)\gamma(\nu)\\
 3&\gamma(\zeta+\omega_2)\gamma(\eta)\gamma(\xi+\Omega)\gamma(\nu)\\
 4&\gamma(\zeta+\omega_2)\gamma(\eta+\Omega)\gamma(\xi)\gamma(\nu).
\end{array}
\]
The auxiliary function
\[
 P(z)=e^{2\pi iZz/\Omega}
       \gamma(\zeta)\gamma(\xi)\gamma(\eta+\Omega)\gamma(\nu+\omega_2)
\]
satisfies the exact meromorphic identity
\begin{equation}
 \sum_{i=1}^4\kappa_iT_i(z)=P(z+\omega_2)-P(z).
 \label{an:primitive}
\end{equation}
To verify it, divide by
$-8e^{2\pi iZz/\Omega}\gamma(\zeta)\gamma(\eta)
\gamma(\xi)\gamma(\nu)$.  The result is the sine identity
\begin{align*}
 &\kappa_1a(\eta)b(\eta)a(\xi)
 +\kappa_2a(\zeta)b(\zeta)a(\xi)
 +\kappa_3a(\zeta)a(\xi)b(\xi)
 +\kappa_4a(\zeta)a(\eta)b(\eta)\\
 &\hspace{15mm}=-b(\eta)
             \{\lambda a(\zeta)a(\xi)+a(\eta)a(\nu)\}.
\end{align*}
The coefficient of $\cos(\pi\eta/\omega_2)$ vanishes because
$\kappa_2b(c)+\kappa_3b(E)=0$.  The other coefficient reduces,
by sine addition, to
$\kappa_1a(\xi)+\kappa_4a(\zeta)=-a(\nu)$.

\emph{Contour justification.}
There is an initial domain where integration of
\eqref{an:primitive} has no residue term.  Take
\begin{gather*}
 u=\Omega/2+\epsilon,\quad v=\omega_2/2+\epsilon,\quad
 Y=\omega_1/2+\eta_0,\quad \Re z=-n/2,\\
 \epsilon,\eta_0\in\R,\qquad
 0<\epsilon<\omega_1/4,\quad0<|\eta_0|<\min(\epsilon,\omega_2).
\end{gather*}
Here $D=\eta_0$, so both sine denominators are nonzero.
On this line the real parts of $\zeta,\eta,\xi,\nu$ are
$\epsilon,\epsilon,\epsilon+\eta_0,\epsilon-\eta_0$, so
every table entry has a separating vertical contour.  The first
right poles of $P$ have distances $\Omega+\epsilon$ and
$\omega_2+\epsilon-\eta_0$ from this line, both greater than $\omega_2$.
The shift of its line to the right by $\omega_2$ therefore crosses no pole.
For $|\Re Z|<(\omega_1-4\epsilon)/2$ the horizontal boundaries decay
exponentially.  Integrating \eqref{an:primitive} gives zero.
The strict inequalities persist in an open complex neighborhood,
which supplies an initial open set for the identity theorem.

The four table integrals, with measure $dz/(i\sqrt\Omega)$, equal
\begin{gather*}
 e^{2\pi iZ\omega_2/\Omega}K_4(R,X,Y+\omega_2,Z),\qquad
 e^{-2\pi iZ\omega_1/\Omega}K_4(R,X,Y-\omega_1,Z),\\
 e^{-\pi iZ\omega_1/\Omega}K_4(R,X-\omega_1,Y,Z),\qquad
 e^{\pi iZ\omega_2/\Omega}K_4(R,X+\omega_2,Y,Z).
\end{gather*}
The respective substitutions in the centered kernels are
$z\mapsto z-\omega_2,z+\omega_1,z+\omega_1/2,z-\omega_2/2$.  Thus the integrated identity
holds on this initial open set with the specified contours.
To continue it, all four products and $P$ have the same total parameter sum
$\Sigma=2u+2v$.  Their common convergence domain is the open convex set
\[
 |\Re Z|<\Omega-\Re(u+v).
\]
Indeed, the general complex-shift asymptotics give the locally uniform
bound
\[
 O\!\left(e^{\pi(\Re\Sigma-2\Omega+2|\Re Z|)|\Im z|/\Omega}\right).
\]
Remove the locally finite affine hyperplanes on which an opposite
pair in one of the four products pinches, or
$D\in\omega_1\Z\cup\omega_2\Z$.  The remaining domain is connected:
a complex line joining two admissible points is contained in none of
these hyperplanes, and its convex open slice remains connected after
removing a locally finite set of points.  The corrected contour
integrals and the coefficients $\kappa_i$ are holomorphic there.
The identity theorem extends the integrated relation throughout this
common convergent, nonpinched domain.

\emph{Specialization.}
At the required parameters,
\[
 c=t_1-\omega_2,\quad D=h-l,\quad E=p-\omega_2,\quad f=p+l-\omega_2.
\]
If $h>l$, substitute
\[
 b(p-\omega_2)=(-1)^pa(p),\quad b(t_1-\omega_2)=(-1)^{t_1}a(t_1),\quad
 b(h-l)=(-1)^{h-l+1}a(h-l)
\]
and multiply the identity by $e^{-2\pi iZ\omega_2/\Omega}a(h-l)$.
Both denominators are nonzero, since $0<h-l<n<\omega_1$ and
$\omega_2$ is irrational.  The three sine conversions use the period
normalization $\omega_1^{-1}+\omega_2^{-1}=1$.
This is \eqref{an:four-term-eq}.  If $h=l$, then $X=Y$;
\eqref{an:P} cancels the first and fourth terms and the second
and third terms separately, without division by $a(h-l)$.

There is no contour pinch at this specialization: the only possibly
negative opposite-parameter sums in the four kernels are
$-h,-l,\omega_2-h,\omega_2-l$.  None equals
$-r\omega_1-t\omega_2$ for $r,t\ge0$, since $\omega_1>n$ and
$\omega_2$ is irrational.  Poles on the same side may merge without
making the continued integral singular.
\end{proof}

\section{Vanishing of the restricted Fourier coefficients}
\label{an:vanishing-section}

\subsection{The negative-column transform}

Choose the principal frequency
$\theta_R\equiv2\pi j/n+\pi(-g-1)\pmod{2\pi}$.
In the range $|\theta_R|<\pi(1-g/\Omega)$ one has
\begin{equation}
 \widehat u_{-g}(j)=-\omega_2 e^{ig\theta_R/2}
  \gamma\left(\frac{\Omega-g}{2}+\frac{\Omega\theta_R}{2\pi}\right)
  \gamma\left(\frac{\Omega-g}{2}-\frac{\Omega\theta_R}{2\pi}\right).
 \label{an:signed-fourier}
\end{equation}
Indeed, use the single-column kernel identity with label $-g$ and
sample on $0,\ldots,n-1$.  Formula \eqref{an:signed} has no correction,
the rectangle has no noninteger pole, and
$-\epsilon+i\R$, $0<\epsilon<1$, already separates the
two gamma pole families.  The horizontal boundaries vanish in the
stated range.  The two-gamma Fourier evaluation and
$2\sin(-\pi g/\omega_1)\gamma(g)=(-1)^gF(-g)$ yield
\eqref{an:signed-fourier}.

\subsection{Cancellation of the transformed residual}

\begin{proposition}[Restricted Fourier vanishing]\label{an:restricted}
Under \eqref{an:ordered}, the Fourier transform of
$\mathcal C_{g,h}(k,l)$ in $k$ vanishes at all $n-g+1$ frequencies
in \eqref{an:restricted-set}.
\end{proposition}

\begin{proof}
Put $\theta_L=-\theta_R$ and $Z=\Omega\theta_L/(2\pi)$.
The restricted frequencies exclude both endpoint ambiguities, so these
choices of principal frequencies agree with the two products in
\eqref{an:finite}.  For the left product put
\[
 r=-g/2,\qquad s_1=-\Omega/2-g/2-h,\qquad
 s_2=\Omega/2-g/2-h,\qquad y=-Z.
\]
Here $Z$ is real and $|Z|<(\Omega-g)/2$, which gives the common
convergence condition for the Euler transformation at $r=-g/2$ and
for the four-term identity at $R=g/2$.  The no-pinch checks in the
product and four-term proofs apply to these kernels: centering,
\eqref{an:P}, and reflection preserve their labelled pole families.
The left product's two integrals are
$e^{\pi ipy/\Omega}K_4(r,s_1,X,y)$ and
$e^{\pi i(n+p)y/\Omega}K_4(r,s_2,X,y)$.
Apply \eqref{an:euler-centered}, then \eqref{an:P}, and reflect
the integration variable.  They become the first and second kernels
of \eqref{an:four-term-eq}.  The two right product integrals are
its third and fourth kernels, with centering factors
$e^{\pi ip'Z/\Omega}$ and $e^{\pi i(n+p')Z/\Omega}$.

Let $\mathcal L(Z)$ denote the four-term sum in
\eqref{an:four-term-eq}.  The product formulas now give
\begin{equation}
\begin{split}
 &\widehat u_{-g}(j)S_{g+h,-h,l-h}(-j)
       -\tau e^{-2\pi ijl/n}S_{-g-l,l,g+h+l}(j)\\
 &\hspace{20mm}=-2\sqrt\Omega\,M\mathcal L(Z),\qquad
 M=\frac{\omega_2(-1)^{g+h(l+1)}F(t_1)}{dF(l)}
                          e^{\pi iZ(\omega_1+l-g)/\Omega}.
\end{split}
 \label{an:coefficients}
\end{equation}
All factors cancelled below are finite and nonzero.  In particular,
$(\Omega-g)/2\pm Z\in(0,\Omega)$, so
$\Gamma_{-g/2}(Z)\ne0$; the integer samples of $F$ are nonzero, and
$h,p,t_1,l\in\{1,\ldots,n-1\}$ makes their sines with denominator
$\omega_1>n$ nonzero.  This also gives $M\ne0$.
To verify the gamma part, the first Euler multiplier is
\[
 \frac{\Gamma_{-g/2}(s_1)\Gamma_{-g/2}(X)}{\Gamma_{-g/2}(-Z)}
 =\frac{F(p)F(t_1)}{F(h)F(l)}
       \frac{\sin(\pi h/\omega_1)}{\sin(\pi p/\omega_1)\Gamma_{-g/2}(Z)}.
\]
The second multiplier divided by the first is
$(-1)^g\sin^2(\pi p/\omega_1)/\sin^2(\pi h/\omega_1)$.
These follow from reflection and the identities, at the nonzero
integer arguments used here,
\[
 \gamma(\omega_1+a)=F(-a),\quad \gamma(\omega_2-a)=1/F(-a),\quad
 2\sin(\pi a/\omega_1)\gamma(-a)=(-1)^aF(a).
\]
The sample factors are
\[
 c_L=\frac{(-1)^{g+hl+1}F(h)}{dF(p)},\qquad
 c_R=\frac{(-1)^{lp+1}F(t_1)}{dF(l)}.
\]
Finally \eqref{an:signed-fourier} reads
$\widehat u_{-g}(j)=-\omega_2 e^{-\pi igZ/\Omega}\Gamma_{-g/2}(Z)$, and
$\theta_L\equiv-2\pi j/n+\pi(g-1)\pmod{2\pi}$
gives the phase relation
\[
 e^{-2\pi ijl/n}e^{-2\pi iZl/\Omega}=(-1)^{l(g+1)}.
\]
For example, after factoring out $-2\sqrt\Omega$, the coefficient of
$K_4(R,X,Y+\omega_2,Z)$ is
\begin{align*}
 &-\omega_2\,
 \frac{(-1)^{g+hl+1}F(h)}{dF(p)}
 (-1)^h\sin\frac{\pi p}{\omega_1}
 \frac{F(p)F(t_1)\sin(\pi h/\omega_1)}
      {F(h)F(l)\sin(\pi p/\omega_1)}
 e^{-\pi iZ(g+p+s_1-X)/\Omega}\\
 &\hspace{25mm}=M\sin\frac{\pi h}{\omega_1},
 \qquad g+p+s_1-X=g-\omega_1-l.
\end{align*}
The other three coefficients follow from the second Euler multiplier,
$c_R$, and the displayed phase relation in the same way.  This proves
\eqref{an:coefficients}, including its normalization.
Lemma~\ref{an:four-term} now makes the bracket in
\eqref{an:finite} zero.
\end{proof}

%% file: proof_revised/cubic_completion.tex
\section{Symmetries and degenerate indices}
\label{sec:degenerate}

The cubic equations have a useful symmetry that becomes visible after
replacing four indices by two triples. Associate to $(g,h,k,l)$ the pair
\begin{equation}\label{tail:triples}
 \boldsymbol a=(0,k-g,l-h),\qquad
 \boldsymbol c=(g+h,-k,-l).
\end{equation}
The pair is \emph{balanced}, meaning that its six entries sum to zero.
Its cross sums are
\begin{equation}\label{tail:crossmatrix}
 (a_i+c_j)_{i,j=0}^2=
 \begin{pmatrix}
 g+h&-k&-l\\
 h+k&-g&k-g-l\\
 g+l&l-h-k&-h
 \end{pmatrix}.
\end{equation}
We call the pair \emph{nondegenerate} when all nine cross sums are nonzero.
Every balanced pair has the form \eqref{tail:triples} after an opposite
translation $(a_i,c_j)\mapsto(a_i-a_0,c_j+a_0)$: take
$k=-c_1$, $l=-c_2$, $g=-c_1-a_1$, and $h=-c_2-a_2$ after translating.

For a balanced pair, set
\[
 S(\boldsymbol a,\boldsymbol c)
 =\sum_m\prod_{i=0}^2W_{m+a_i}\prod_{j=0}^2W_{c_j-m},
 \qquad
 P(\boldsymbol a,\boldsymbol c)=\prod_{i,j=0}^2W_{a_i+c_j}.
\]

\begin{lemma}[Factorization and permutation symmetry]\label{tail:factor-lemma}
For a nondegenerate pair,
\begin{equation}\label{tail:factor}
 \mathcal C_{g,h}(k,l)
 =\frac{W_{-g-h}W_gW_h}{\delta^5}
       \bigl(\delta^2S(\boldsymbol a,\boldsymbol c)
                      -\tau P(\boldsymbol a,\boldsymbol c)\bigr).
\end{equation}
Its vanishing is preserved by independent permutations of the two
triples and by opposite translation.
\end{lemma}

\begin{proof}
Substitute \eqref{tail:uniform-core} in \eqref{tail:C} and use
$W_aW_{-a}=\delta$ for $a\ne0$. Nondegeneracy excludes every origin
correction. The prefactor in \eqref{tail:factor} is nonzero.
Both $S$ and $P$ are invariant under the permutations; opposite
translation is absorbed by reindexing $m$ in $S$.
\end{proof}

\begin{lemma}[Degenerate cubics]\label{tail:collisions}
If any cross sum in \eqref{tail:crossmatrix} is zero, then
$C(g,h,k,l)=0$ follows from $Q=0$, the coefficient formula
\eqref{tail:uniform-core}, $W_0=-1$, the reciprocal product identities,
and $\delta^2=n\delta+1$.
\end{lemma}

\begin{proof}
We first treat $g,h,g+h\ne0$, then the three boundary cases.

\emph{A zero cross sum with $g,h,g+h\ne0$.}
Keeping the origin corrections in the preceding substitution gives
\begin{equation}\label{tail:balanced}
 \mathcal C_{g,h}(k,l)
 =\frac{W_{-g-h}W_gW_h}{\delta^5}
 \left(\delta^2S-\tau P-\tau^3\delta^3\,
 \ind{\{c_0,c_1,c_2\}=\{-a_0,-a_1,-a_2\}}\right).
\end{equation}
Here equality in the indicator means equality of multisets.
The only possible origin corrections occur at
$(k,l)=(0,-g)$ and $(-h,0)$. These are exactly the two possible
multiset matches: a matching permutation must avoid the three
nonzero diagonal cross sums and hence is one of the two derangements.
A repeated row would also allow a forbidden diagonal match, so a full
match has three distinct rows. In that case,
\[
 S=\delta^3(n-3)+3\delta^2,\qquad P=-\delta^3.
\]
Thus the bracket in \eqref{tail:balanced} vanishes by
$n=\delta-\delta^{-1}$ and $\tau=\delta-1$.

If there is no full match, permute the triples so that $a_0+c_0=0$.
Then
\[
 W_{m+a_0}W_{-m-a_0}=\delta-\tau\ind{m=-a_0}.
\]
Put $H=a_1+c_1$, $V=a_2+c_1$. Both are nonzero: either zero
would force the remaining entries to match as well. The unrestricted
sum of the other four factors is
\[
\sum_tW_tW_{H-t}W_{-t-V}W_{t+V-H}
       =\delta\tau^2Q(-V,H)=0.
\]
Indeed, with $m=t+V$, the sum is
$\delta\sum_m\mathsf B_{m-V,H}\mathsf B_{H,m}$.
The uniform coefficient formula has no origin correction in these
entries because $H\ne0$. The substitution uses
$W_HW_{-H}=\delta$, and $H,V\ne0$ remove both indicator terms
in $Q(-V,H)$.
Writing
$T=\prod_{i=1}^2W_{a_i-a_0}\prod_{j=1}^2W_{c_j+a_0}$,
we obtain $S=-\tau T$. The four interior factors of $P$ pair
to $\delta^2$, and $W_0=-1$, so $P=-\delta^2T$.
Again $\delta^2S-\tau P=0$.

\emph{The boundary cases.}
Set $c=\tau^{-1}$. For $h\ne0$, the coefficient formula gives
\begin{equation}\label{tail:two-products}
 \mathsf A_{u,h}\mathsf A_{h,u}=
 \begin{cases}
 c^2,&u=0\text{ or }u=h,\\
 \delta c^2,&u\ne0,h,
 \end{cases}
 \qquad \delta c^2-c=c^2.
\end{equation}
For $g=0$, $h\ne0$, use the fixed row
$\mathsf A_{0,a}=\ind{a=0}-c$ and $Q=0$ to reduce the cubic sum to
\[
 \mathsf A_{-k,h}\mathsf A_{h,l-k}-c\ind{l=0}.
\]
For $l\ne0$, the uniform formula has no origin corrections and gives
\[
 \mathsf B_{-k,h}\mathsf B_{h,l-k}
 =\delta^{-1}W_{-k}W_{h+k}W_{l-k-h}W_{k-l}
 =\mathsf B_{l,k}\mathsf B_{h+k,l}.
\]
Dividing by $\tau^2$ gives the required equality.
For $l=0$, the target product is
$(\ind{k=0}-c)(\ind{k=-h}-c)$.
Both sides equal $c^2-c$ at $k=0,-h$, and $c^2$ elsewhere,
by \eqref{tail:two-products}.
The case $h=0$, $g\ne0$ follows by interchanging $(g,k)$ and $(h,l)$.
For $h=-g\ne0$, the same quadratic reduction gives
\[
 \mathsf A_{g,k}\mathsf A_{-l,g}-c\ind{k=g+l}.
\]
Substitution treats $k\ne g+l$, and \eqref{tail:two-products}
treats $k=g+l$.

Finally, when $g=h=0$, the cubic sum is
\[
 \ind{k=l=0}
 -c\bigl(\ind{k=0}+\ind{l=0}+\ind{k=l}\bigr)
 +3c^2-nc^3.
\]
The product $\mathsf A_{l,k}\mathsf A_{k,l}$ equals
$(1-c)^2$, $c^2$, or $\delta c^2$, according as all, exactly one,
or none of the equalities $k=0$, $l=0$, $k=l$ hold.
Using $\delta^2=n\delta+1$ in each case proves that the displayed
sum is $\mathsf A_{l,k}\mathsf A_{k,l}-\delta^{-1}$, as required.
\end{proof}

\section{Polynomial interpolation and index reduction}
\label{tail:completion}

The contour calculation gives only a range of Fourier coefficients.
The permutation symmetry now supplies enough zeros to recover the rest.

\subsection{Induction in the ordered range}

\begin{proposition}\label{tail:ordered}
For every odd $n\ge3$, the cubics vanish in the following ordered range:
\[
 \mathcal C_{g,h}(k,l)=0\quad
 (g,h\ge1,\ g+h<n,\ 1\le l\le h,\ k\in G).
\]
\end{proposition}

\begin{proof}
Fix $n$ and use strong induction on the positive integer $g$. By
Proposition~\ref{an:restricted}, for each ordered $g,h,l$, precisely
$n-g+1$ of the $n$ Fourier coefficients of the sequence
$k\mapsto\mathcal C_{g,h}(k,l)$ are known to vanish.
The omitted frequencies form a consecutive cyclic block of
length $g-1$: the principal integers $v=n\theta_R/\pi$,
of parity $g+1$, advance by two as the Fourier index advances
by one, and the frequency condition is $|v|\le n-g$.
For $g=1$ there are no omitted frequencies, proving the base case.
For $g>1$, Fourier inversion therefore gives, with
$\zeta=\exp(2\pi i/n)$,
\begin{equation}\label{tail:polynomial}
 \mathcal C_{g,h}(k,l)=\zeta^{ak}f(\zeta^k),\qquad
                         \deg f\le g-2
\end{equation}
for an integer $a$ and a complex polynomial $f$.
More explicitly, if the omitted block starts at $a$, Fourier inversion
gives
\[
 f(X)=\frac1n\sum_{j=0}^{g-2}
          \widehat{\mathcal C}_{g,h}(a+j,l)X^j,
\]
where the hat denotes the Fourier transform in the first argument $k$.

We show that $k=1,\ldots,g-1$ are zeros.  A degenerate index choice is
already covered by Lemma~\ref{tail:collisions}.  Otherwise
perform the following changes of cubic labels:
\begin{equation}\label{tail:descent}
 (g,h,k,l)\longmapsto
 \begin{cases}
 (k,h,g,l+g-k),&l+g-k\le h,\\
 (k,n-h-k,k+h-l,n-g-l),&l+g-k>h.
 \end{cases}
\end{equation}
These changes arise from permutations and opposite translation
of \eqref{tail:triples}.  With positions numbered $0,1,2$,
the first uses rows $(1,0,2)$ and unchanged columns; the second
uses rows $(2,0,1)$ and columns $(2,1,0)$.  In both cases
translate the first row entry to zero.  Substitution gives
exactly \eqref{tail:descent}, with all entries interpreted
modulo $n$ where necessary.

The new first index is $k<g$.  Both new quadruples are ordered.
This is immediate in the first case; in the second,
\[
 k+(n-h-k)=n-h<n,\qquad
 1\le n-g-l\le n-h-k,
\]
where the last inequality follows from $l+g-k>h$.
The induction hypothesis and the nonzero factor in
\eqref{tail:factor} therefore give each required zero of the
original cubic.  The $g-1$ points
$\zeta,\ldots,\zeta^{g-1}$ are distinct, including for
composite $n$, so \eqref{tail:polynomial} forces $f=0$.
\end{proof}

\subsection{Putting the indices in the ordered range}

\begin{lemma}\label{tail:normalform}
Let $X_{ij}\in\{1,\ldots,n-1\}$ represent the cross sums of a
nondegenerate balanced pair of triples. If
$X_{00}+X_{11}+X_{22}=2n$, independent permutations and opposite
translation put the pair in the ordered range of
Proposition~\ref{tail:ordered}.
\end{lemma}

\begin{proof}
There exist $i\ne j$ with $X_{ij}\ge X_{jj}$.  Otherwise
each diagonal entry would be a strict maximum in its column.
The three row labels would then be distinct.  Let $d_j>0$
be the forward cyclic gap from $a_j$ to the next row label.
Since the other entry in column $j$ is smaller and no cross
sum is zero, $X_{jj}+d_j>n$.  The gaps sum to $n$, so summing
would give $\sum_jX_{jj}>2n$, a contradiction.

Choose such $i,j$, let $r$ be the remaining position, and use
row and column order $(i,r,j)$.  Normalize the first row to
zero.  The resulting cubic indices have
\[
 g=n-X_{rr},\qquad h=n-X_{jj},\qquad l=n-X_{ij}.
\]
Thus $g,h\ge1$, $1\le l\le h$, and
$g+h=2n-X_{rr}-X_{jj}=X_{ii}<n$, as required.
\end{proof}

\section{Reflection by matrix inversion}
\label{sec:reflection}

Matrix inversion transfers a complete cubic block with first indices
$(g,h)$ to the block with first indices $(-h,-g)$. This is the final
step needed to cover every index choice.

\begin{lemma}\label{allodd:reflection}
Suppose the matrix \eqref{foundation:core} satisfies the quadratic equations
and $W_aW_{-a}=\delta$ for $a\ne0$.  Fix $g,h\in G$, put $t=g+h$, and assume $g,h,t\ne0$.  If $C(g,h,k,l)=0$ for every $k,l$, then
$C(-h,-g,k,l)=0$ for every $k,l$.
\end{lemma}

\begin{proof}
We invert the two sides of the cubic block separately. All matrices
have rows and columns indexed by $G$.

\emph{1. Column inverses.} Write $u_p(a)=\mathsf B_{a,p}$ for
$p\ne0$.  The coefficient formula and the quadratic equations give
\begin{equation}\label{allodd:column-inverses}
 u_p(a)=\frac{W_aW_{p-a}}{W_p}=u_p(p-a),\qquad
 \sum_a u_p(a)u_{-p}(x-a)=\tau^2\ind{x=0}.
\end{equation}
For the second identity, multiply $Q(x,p)=0$ by $\tau^2$,
use $\mathsf B_{p,a}=u_{-p}(a-p)=u_{-p}(-a)$, and reindex.
Consequently the circulant with entries $u_p(k-l)$ has inverse
$\tau^{-2}u_{-p}(k-l)$.

\emph{2. The left side.} Define
\[
 L_{g,h}(k,l)=\sum_m u_t(m)\mathsf B_{g,m+k}
                                      \mathsf B_{h,m+l},
 \qquad
 R_{g,h}(k,l)=\mathsf B_{g+l,k}\mathsf B_{h+k,l}.
\]
The assumed cubics say $L_{g,h}=\tau R_{g,h}$.
Let $H_p(k,m)=\mathsf B_{p,k+m}$ and let $D_t$ be the
diagonal matrix with entries $u_t(m)$.  Equation
\eqref{allodd:column-inverses} gives
\[
 H_p^{-1}(k,m)=\tau^{-2}u_p(k+m),\qquad
 L_{g,h}=H_gD_tH_h.
\]
All entries of $D_t$ are nonzero.  Thus $L_{g,h}$ is invertible.
The reciprocal product identities, including the two exceptional samples $m=0,t$,
give
\begin{equation}\label{allodd:diagonal-inverse}
 \frac1{u_t(m)}=\delta^{-1}u_{-t}(-m)
       -\frac{\tau}{\delta}
                    (\ind{m=0}+\ind{m=t}).
\end{equation}
Set
\[
 E(k,l)=u_h(k)u_g(l)+u_h(k+t)u_g(l+t),
 \qquad
 L^{\mathrm r}(k,l)=L_{-h,-g}(-k,-l).
\]
On inserting \eqref{allodd:diagonal-inverse} into
$L_{g,h}^{-1}=H_h^{-1}D_t^{-1}H_g^{-1}$ and replacing
$m$ by $-m$, we obtain
\begin{equation}\label{allodd:left-inverse}
 L_{g,h}^{-1}=\delta^{-1}\tau^{-4}
                         (L^{\mathrm r}-\tau E).
\end{equation}
Here $u_h(k-m)=\mathsf B_{-h,m-k}$, by
\eqref{allodd:column-inverses} and order-three symmetry.

\emph{3. The right side.} We prove, using only the quadratics and
reciprocal products, the corresponding identity
\begin{equation}\label{allodd:right-inverse}
 R^{\mathrm r}=\delta\tau^2 R_{g,h}^{-1}+E,
 \qquad
 R^{\mathrm r}(k,l)=R_{-h,-g}(-k,-l).
\end{equation}
The details below retain the origin corrections in the coefficient formula.
Put
\[
 d_1(k)=W_{h+k}W_{-k},\qquad d_2(l)=W_{g+l}W_{-l},
 \qquad b=u_t(g)=u_t(h)=\frac{W_gW_h}{W_t}\ne0.
\]
Direct substitution of
$\mathsf B_{a,b}=\delta^{-1}W_aW_{b-a}W_{-b}
+\tau^2\delta^{-1}\ind{a=b=0}$ yields
\begin{equation}\label{allodd:right-factor}
 R_{g,h}=\delta^{-2}W_{-t}\,
                   \operatorname{diag}(d_1)M\operatorname{diag}(d_2),
\end{equation}
where
\[
 M(k,l)=u_{-t}(k-l-g)
       +\tau^2(\ind{k=0,l=-g}+\ind{k=-h,l=0}).
\]
The uncorrected matrix $K(k,l)=u_{-t}(k-l-g)$ has inverse
$K^{-1}(k,l)=\tau^{-2}u_t(k-l+g)$.
Writing $U=(e_0,e_{-h})$ and $V=(e_{-g},e_0)$, we have
$M=K+\tau^2 UV^{\mathsf T}$ and
\[
 I_2+\tau^2V^{\mathsf T}K^{-1}U=
       \begin{pmatrix}0&b\\b&0\end{pmatrix}.
\]
This matrix is invertible.  The rank-two inverse formula therefore
proves that $M$, hence $R_{g,h}$, is invertible and that
\begin{equation}\label{allodd:woodbury}
 M^{-1}(k,l)=\tau^{-2}(U_0-U_1-U_2),
\end{equation}
where, for the entry $(k,l)$ in question,
\[
 U_0=u_t(k-l+g),\qquad
 U_1=b^{-1}u_t(k+g)u_t(g-l),\qquad
 U_2=b^{-1}u_t(k+t)u_t(-l).
\]

To check the exceptional entries in
\eqref{allodd:right-inverse}, define
\[
 q_a=W_aW_{-a}=\delta-\tau\ind{a=0},\quad
 \alpha=\frac{q_kq_l}{\delta^2},\quad
 \beta=\frac{q_{g+k}q_{h+l}}{\delta^2},\quad
 Z=\frac{W_{-t}d_2(k)d_1(l)}{\delta^3}.
\]
To see the endpoint corrections explicitly, expand
$R^{\mathrm r}(k,l)=\mathsf B_{-h-l,-k}\mathsf B_{-g-k,-l}$
using \eqref{tail:uniform-core}. The reciprocal products give
\[
\begin{aligned}
 ZR^{\mathrm r}(k,l)
   &=\alpha\beta U_0+\frac{\tau^2}{\delta^2}
       (\ind{k=0,l=-h}+\ind{k=-g,l=0}),\\
 Z\,u_h(k)u_g(l)&=\alpha U_1,\\
 Z\,u_h(k+t)u_g(l+t)&=\beta U_2.
\end{aligned}
\]
In the first line the two corrections come from the two possible
origin entries of $\mathsf B$; they cannot occur together since
$g,h\ne0$. The other two lines follow from $W_tW_{-t}=\delta$
and the definitions of $U_1,U_2$.
Subtracting the two terms of $E$ gives the exact entry identity
\begin{equation}\label{allodd:endpoints}
 Z(R^{\mathrm r}-E)(k,l)
  =\alpha\beta U_0-\alpha U_1-\beta U_2
     +\frac{\tau^2}{\delta^2}
        (\ind{k=0,l=-h}+\ind{k=-g,l=0}).
\end{equation}
If $\alpha\ne1$, then $k=0$ or $l=0$, and $U_1=U_0$.
If $\beta\ne1$, then $k=-g$ or $l=-h$, and $U_2=U_0$.
Both exceptional conditions can hold only at $(k,l)=(0,-h)$
or $(-g,0)$: the assumptions $g,h\ne0$ exclude the other
intersections.  At either point,
$\alpha=\beta=\delta^{-1}$ and $U_0=U_1=U_2=-1$.
Thus the last term of \eqref{allodd:endpoints} cancels
$(\alpha-1)(\beta-1)U_0=-\tau^2/\delta^2$.
In every case the right side of \eqref{allodd:endpoints} is
$U_0-U_1-U_2$.  Equations \eqref{allodd:right-factor} and
\eqref{allodd:woodbury} now give
\[
 Z(R^{\mathrm r}-E)(k,l)
   =\tau^2M^{-1}(k,l)
   =Z\delta\tau^2R_{g,h}^{-1}(k,l).
\]
The factor $Z$ is nonzero, proving
\eqref{allodd:right-inverse} at all entries.

\emph{4. Comparing the inverses.} Now $L_{g,h}=\tau R_{g,h}$ and
\eqref{allodd:left-inverse} imply
\[
 L^{\mathrm r}=\delta\tau^3R_{g,h}^{-1}+\tau E
             =\tau R^{\mathrm r},
\]
where the second equality is \eqref{allodd:right-inverse}.
These are precisely all the cubics with first indices $(-h,-g)$.
\end{proof}

\begin{proposition}\label{allodd:all-cubics}
For every odd $n\ge3$, the analytic coefficient matrix satisfies
$C(g,h,k,l)=0$ for all $g,h,k,l$.
\end{proposition}

\begin{proof}
The degenerate indices, including the cases $g=0$, $h=0$, or $g+h=0$ have
already been proved in Lemma~\ref{tail:collisions}.
Take positive representatives $1\le g,h<n$ with $g+h<n$.
For every nondegenerate pair associated with $(g,h,k,l)$, the
distinguished cross diagonal has representatives
\[
 (g+h,n-g,n-h),
\]
whose sum is $2n$.  The proof of Lemma~\ref{tail:normalform}
therefore puts this pair into the ordered range.  Proposition~\ref{tail:ordered} and
the permutation and translation invariance in
Lemma~\ref{tail:factor-lemma} prove all these cubics, for every $k,l$.

If instead $g+h>n$, the positive representatives
$g'=n-h$, $h'=n-g$ satisfy $g'+h'<n$.  All cubics for $(g',h')$
have just been proved; Lemma~\ref{allodd:reflection} then gives
all cubics for $(-h',-g')=(g,h)$.  This exhausts the indices.
\end{proof}

%% file: proof_revised/reconstruction.tex
\section{From the matrix equations to fusion categories}
\label{tail:reconstruction}

The preceding parts prove the linear, quadratic, and cubic equations.
Only the quartic equations remain before reconstruction can be applied.

\begin{lemma}[The quartics follow from the cubics]\label{tail:quartic-lemma}
The matrix $\mathsf A$ satisfies the quartic Evans--Gannon equations.
\end{lemma}

\begin{proof}
Write their quartic residual as
\begin{align*}
 K(g,h,i,k)={}&\sum_{l,m}\mathsf A_{l,m}\mathsf A_{l+g,h}
                  \mathsf A_{h+m,l+i}\mathsf A_{i,k+m}
       -\mathsf A_{h-g,i-g}\ind{k=g}\\
 &\quad+\delta^{-1}\ind{h=0}\mathsf A_{i,k}
       +\delta^{-1}\mathsf A_{g,h}\ind{i=0}.
\end{align*}
Order-three symmetry gives the polynomial identity
\begin{equation}\label{tail:quartic}
 K(g,h,i,k)=\sum_m\mathsf A_{i,k+m}C(-h,h+m,g-h,i)
             +\mathsf A_{i-h,g-h}Q(g-k,i).
\end{equation}
Indeed, replace $\mathsf A_{l+g,h}$ in the double sum by
$\mathsf A_{-h,l+g-h}$. The inner sum becomes
\[
 C(-h,h+m,g-h,i)+\mathsf A_{i-h,g-h}\mathsf A_{m+g,i}
                  -\delta^{-1}\ind{h=0}\ind{m=0}.
\]
Multiply by $\mathsf A_{i,k+m}$ and sum over $m$.
The remaining quadratic sum is
$Q(g-k,i)+\ind{g=k}-\delta^{-1}\ind{i=0}$.
Its correction terms cancel by order-three symmetry:
$\mathsf A_{i-h,g-h}=\mathsf A_{h-g,i-g}$ and
$\mathsf A_{-h,g-h}=\mathsf A_{g,h}$.
This proves \eqref{tail:quartic}, and $Q=C=0$ gives $K=0$.
\end{proof}

\begin{lemma}[Changing the dimension embedding]\label{tail:embedding}
There is a field automorphism $\sigma$ of $\C$ with $\sigma(\delta)=d$.
Applying it entrywise to $\mathsf A$ gives a solution of the same
reconstruction equations with dimension parameter $d$.
\end{lemma}

\begin{proof}
The two roots of $X^2-nX-1$ are $\delta$ and $d$, and its
discriminant $n^2+4$ is not a square for $n\ge3$.
Thus $\Q(\delta)$ has an automorphism interchanging the roots.
Extend it to $\overline{\Q}$, fix a transcendence basis of
$\C/\overline{\Q}$, and extend to the algebraic closure $\C$.
This gives $\sigma$. All reconstruction identities are polynomial
identities over $\Q$ after their nonzero denominators are cleared,
so they are preserved by $\sigma$.
\end{proof}

\begin{proof}[Proof of Theorem~\ref{intro:main}]
The linear identities \eqref{tail:linear}, the quadratics
(Proposition~\ref{foundation:quadratics}), and the quartics
(Lemma~\ref{tail:quartic-lemma}) give exactly
\cite[Eqs.~(4.7)--(4.10)]{EG} with $\omega=1$.
Theorem~2(a) of that paper allows the negative root $\delta$ and
constructs a spherical fusion category with simple objects
$\alpha^a$ and $\rho_a=\alpha^a\rho$, $a\in G$,
the fusion rules \eqref{intro:fusion1}--\eqref{intro:fusion2}, and
categorical dimensions $1$ and $\delta$.

Apply Lemma~\ref{tail:embedding} and reconstruct once more.
The resulting spherical category has categorical dimensions $1$ and $d$,
equal to its Frobenius--Perron dimensions. Its global dimension is
$n(1+d^2)$, also its Frobenius--Perron dimension, so it is pseudo-unitary.
Finally, the categories for different odd $n$ have different ranks $2n$,
and hence are pairwise inequivalent.
\end{proof}

\begin{remark}
Pseudo-unitarity is the conclusion here. Unitarity would additionally
require a Hermitian positive-dimensional reconstruction matrix
\cite[Theorem~2(c)]{EG}; the argument above does not establish that
property. Neither arithmetic information about the coefficient field
nor a second reflected contour calculation is needed for the stated
existence theorem.
\end{remark}